\documentclass[11pt]{article}

\usepackage[T1]{fontenc}
\usepackage[utf8]{inputenc}
\usepackage{lmodern}
\usepackage{amsmath,amssymb,amsthm,mathtools}
\usepackage{booktabs,array,tabularx,longtable}
\usepackage{graphicx}
\usepackage{tikz}
\usetikzlibrary{arrows.meta,calc,decorations.pathreplacing,positioning}
\usepackage[margin=1in]{geometry}
\usepackage{microtype}
\usepackage{xcolor}
\usepackage{enumitem}
\usepackage{url}
\usepackage[hidelinks]{hyperref}
\usepackage[capitalise,noabbrev]{cleveref}

\newtheorem{theorem}{Theorem}[section]
\newtheorem{proposition}[theorem]{Proposition}
\newtheorem{lemma}[theorem]{Lemma}
\newtheorem{corollary}[theorem]{Corollary}
\theoremstyle{definition}
\newtheorem{definition}[theorem]{Definition}
\newtheorem{example}[theorem]{Example}
\theoremstyle{remark}
\newtheorem{remark}[theorem]{Remark}
\newtheorem{conjecture}[theorem]{Conjecture}
\newtheorem{question}[theorem]{Question}

\newcommand{\R}{\mathbb R}
\newcommand{\C}{\mathbb C}
\newcommand{\one}{\mathbf 1}
\newcommand{\trans}{T_{\mathrm{tr}}}
\newcommand{\Ksh}{K_{\mathrm{sh}}}
\newcommand{\Gsh}{G_{\mathrm{sh}}}
\newcommand{\Ghyb}{G_{\mathrm{hyb}}}

\newcommand{\rank}{\operatorname{rank}}
\newcommand{\im}{\operatorname{im}}
\newcommand{\Span}{\operatorname{span}}
\newcommand{\Pf}{\operatorname{Pf}}
\newcommand{\relint}{\operatorname{relint}}
\newcommand{\Tight}{\mathcal T}
\newcommand{\KKT}{\mathcal K}
\newcommand{\Area}{\operatorname{Area}}

\title{KKT Stresses, Affine Moments, and Separator Flux\\
in the Heilbronn Triangle Problem}
\author{Dawid Trela\\[0.35em]
\small Doctoral School and Faculty of Law and Administration, War Studies University, Warsaw, Poland\\
\small \texttt{dawidmtrela@gmail.com}\\
\small ORCID: \href{https://orcid.org/0000-0001-9781-6425}{0000-0001-9781-6425}}
\date{24 August 2026}

\hypersetup{
  pdftitle={KKT Stresses, Affine Moments, and Separator Flux in the Heilbronn Triangle Problem},
  pdfauthor={Dawid Trela},
  pdfsubject={Discrete and computational geometry},
  pdfkeywords={Heilbronn triangle problem, KKT stress, affine moment identity, shared kernel, separator flux, exact symbolic computation}
}

\begin{document}
\maketitle

\begin{abstract}
For \(n\) points in the unit square, the Heilbronn triangle problem asks for
the largest possible minimum triangle area.  We develop a variational stress
theory for this max--min problem.  At every positive-area local optimum,
normalized Karush--Kuhn--Tucker multipliers assemble into a skew matrix
\(B\) satisfying \(Bz=2ib\), where \(b\) is the outward square reaction.
This equilibrium has the isotropic affine moment
\(\sum_i p_i b_i^T=\Delta I_2\); the identity also holds for every tight
subfamily carrying weights inherited from the same multiplier, with
\(\Delta\) scaled by its multiplier mass.  It follows that every positive
stress component meets all four sides and that there are at most two such
components.

To handle nonunique multipliers, we introduce the intersection of the stress
kernels over the whole KKT face and a canonical hybrid operator incorporating
all tight determinant and boundary derivatives.  A strictly convex selector
removes every decomposable invisible motion, leaving a rank-one-free residual
with a sharp dimension bound; the literal maximal two-dimensional residual
is excluded.  Two-terminal substresses satisfy an exact interface-flux law,
and a five-internal-vertex rank-four block has a Pfaffian cofactor carrier.
Finally, an analytic
reduction followed by exact symbolic enumeration excludes every
one-external completion of a specified one-silent five-cycle residual, for
all orientation words.  These results isolate the remaining degeneracy but
do not solve the problem for arbitrary \(n\).
\end{abstract}



\tableofcontents

\section{Introduction}
\label{sec:introduction}

Let \(P=\{p_1,\ldots,p_n\}\) be a set of points in the unit square.  The
square Heilbronn triangle problem asks how large one can make the smallest
area of a triangle determined by three points of \(P\).  In the notation
fixed below, the extremal value is
\[
 \Delta_n=\max_{P\in[0,1]^{2n}}
 \min_{1\le i<j<k\le n}\Area(p_i,p_j,p_k).
\]
The question goes back to Heilbronn and has generated two rather different
lines of work.  One seeks asymptotic bounds as \(n\) grows; the other
determines extremal configurations for fixed, small \(n\)
\cite{Roth1951,Zakharov2026}.  For the square, exact values have been proved
through \(n=8\)
\cite{YangZhangZeng1991,DressYangZeng1995,ZengChen2011,DehbiZeng2022}.
The nine-point record has recently crossed from numerical construction to a
more delicate form of certification: an exact algebraic configuration is
paired with an \(\varepsilon\)-global numerical upper-bound certificate, but
no independent exact symbolic global upper-bound proof is presently known
\cite{ComellasYebra2002,ChenXuZeng2017,MonjiModirKocuk2025,SudermannMerx2026}.
The latest audited source treats certification for \(n=10\) as future work
\cite{SudermannMerx2026}.  The asymptotic theory is farther removed from
explicit configurations.  Its history includes the work of Roth and of
Koml\'os, Pintz, and Szemer\'edi, while the current upper-bound line gives
\(\Delta_n\lesssim_\varepsilon n^{-7/6+\varepsilon}\)
\cite{Roth1951,KomlosPintzSzemeredi1981,KomlosPintzSzemeredi1982,
CohenPohoataZakharov2025,Zakharov2026}.

This paper studies the local structure forced at an extremal configuration.
The starting point is elementary: in a fixed orientation cell, the minimum
area conditions are smooth determinant inequalities, so their
Karush--Kuhn--Tucker multipliers can be read as weights on oriented triangles.
The determinant gradients make this interpretation considerably more rigid
than a generic KKT reformulation.  Each weighted triangle contributes a
rank-two skew matrix, and their sum is an equilibrium stress on the vertex
set.  The reaction of the four sides of the square is then encoded by the
single complex equation
\begin{equation}
  Bz=2ib. \label{eq:intro-equilibrium}
\end{equation}
Here \(z_i=x_i+iy_i\), and \(b_i\) is the outward reaction at \(p_i\).  The
normalization in \eqref{eq:intro-equilibrium} uses ordinary, rather than
double, area.

Two affine moment identities emerge from this equilibrium.  For the full
stress they are
\begin{equation}
  \sum_i b_i=0,
  \qquad
  \sum_i p_i b_i^T=\Delta I_2.        \label{eq:intro-noether}
\end{equation}
More surprisingly, the second identity remains true for an arbitrary
positive subfamily of active triangles: the right-hand side is simply
multiplied by the mass of that subfamily.  We call these Noether-type
identities because translations and infinitesimal linear changes of
coordinates are the underlying variations.  Nothing in their proof invokes
a conservation law in time; \emph{affine moment identity} is an equally
accurate description.

The KKT multiplier is usually not unique.  Consequently, the nullspace of one
stress is not intrinsic to the configuration.  The correct object is the
intersection over the entire multiplier face.  We show that this shared
kernel is realized by finitely many actual KKT stresses and is the kernel of
a canonical positive semidefinite operator.  Combining this operator with
the row space of all tight determinant and boundary derivatives gives a
hybrid observability operator.  Its kernel is exactly the common part of the
stress fibre and the tight-constraint tangent space.  Thus the failure of
first-order reconstruction is not represented by a rank test chosen from one
multiplier; it is a geometric residual attached to the whole KKT face.

The residual has a useful terminal form.  On the compact optimal level set,
maximize the strictly convex function
\(
  \Phi(P)=\sum_i\|p_i\|^2.
\)
At such a representative no invisible motion can move all vertices in
parallel with vertex-dependent speeds.  Indeed, determinant areas are exactly
quadratic, and the quadratic term of such a motion vanishes.  A surviving
residual is therefore a real matrix space containing no nonzero rank-one
matrix.  This observation yields a graph-of-an-operator normal form and the
sharp bound
\[
 \dim W\le 2\left\lfloor\frac{\dim H}{2}\right\rfloor,
 \qquad H=\Ksh/\langle\one\rangle.
\]
It also allows a complete exclusion of the maximal two-dimensional residual.

Partial stresses reveal a second layer of structure.  Suppose a positive
substress meets the remainder of the positive support at precisely two
vertices, labelled \(a\) and \(b\).  If a shared scalar coordinate \(\xi\)
has zero internal residual, skew symmetry leaves a single interface flux
\(\rho\):
\[
 B_{\mathcal H}\xi=\rho(e_a-e_b).
\]
The corresponding partial load satisfies
\begin{equation}
 \sum_i\xi_i(b_i^{\mathcal H})^T
 =\frac{\rho}{2}
   (Y_b-Y_a,\,X_a-X_b).               \label{eq:intro-flux}
\end{equation}
Thus interface flux couples the shared coordinate to a right-angle rotation
of the physical separator vector.  Nonzero flux is carried by a zero mode of
the internal skew block.  When that block has five vertices and rank four,
its kernel is generated by the Pfaffian cofactor vector.  The same moment
identity shows that an internally isolated two-terminal block must consume an
internal square reaction; hence there are at most eight pairwise internally
disjoint, internally isolated blocks of this kind.  For one fixed terminal
pair and one fixed shared scalar coordinate, at most four can have nonzero
flux.  These are block counts, not bounds on the size or depth of a block.

The final part of the paper uses exact symbolic computation to probe a minimal
residual that survives all local determinant equations.  It consists of one
stress-silent vertex and five tight but unsupported triangles forming an
induced five-cycle.  The six-label area-and-trace system has a smooth real
four-dimensional local component, so it cannot be dismissed as a formal
infinitesimal artefact.  Nevertheless, for every one of the \(32\) orientation
words, no completion using only one additional positive-support label exists
under the full stated hypotheses.  The reduction is analytic; the remaining
finite cases are checked exactly, including singular charts and inactive
inequalities.  The calculation is supplied in a self-contained supplement.

\subsection*{Relation to rigidity and determinant geometry}

Stress matrices are classical in rigidity theory, and recent determinant
hypergraph work uses common stress kernels to study identifiability
\cite{GortlerHealyThurston2010,CruickshankEtAl2024}.  Those results concern
generic equality frameworks and symmetric or signed weighted-adjacency
operators.  Here the stress is skew, its coefficients form a nonnegative KKT
face, and the square contributes an isotropic reaction moment.  Thus the
existing common-kernel theorems provide a close analogy but do not imply the
support or observability statements proved below.

Small separators, 2-sums, and pinned reaction systems also have well-developed
graph-rigidity counterparts
\cite{ServatiusServatius2011,ServatiusShaiWhiteley2008,Garamvolgyi2025,
MalicStreinu2023}.  Their hypotheses use graph edges, genericity, or a shared
edge, whereas our interface belongs to a triangle hypergraph and consists of
two shared labels.  In particular, these theories do not yield the partial
affine-moment reaction or the numerical block budgets in
\Cref{sec:separator}; indeed, 2-sum and pinned families warn that a fixed
interface need not bound the order or depth of a block.  Maxwell--Cremona
correspondences motivate some equilibrium language but use planar axial edge
stresses rather than the present skew triangle operator
\cite{EricksonLin2022}.  We therefore prove all stress, moment, and separator
identities directly.

\subsection*{Main results}

The exact formulations are distributed through the paper.  The following
summary gives the logical spine.

\begin{itemize}[leftmargin=2em]
\item \textbf{Stress and affine moments.}
Every positive-area local optimum admits a normalized nonnegative triangle
stress satisfying \eqref{eq:intro-equilibrium}.  Both the full stress and
every positive substress satisfy isotropic moment identities
(\Cref{thm:stress-noether,thm:partial-noether}).

\item \textbf{Support and boundary geometry.}
Every positive stress component meets all four sides of the square, and there
are at most two such components (\Cref{thm:components}).  Tight constraints
need not lie in the maximal support of the multiplier face; this distinction
is retained throughout.

\item \textbf{Shared-kernel observability.}
The whole KKT face has a finite actual stress realization and a canonical
Gram operator.  Adding the projector onto all tight rows gives an operator
whose kernel is precisely the invisible shared residual
(\Cref{thm:shared-gram,thm:hybrid-kernel}).

\item \textbf{Terminal residuals.}
An optimal level set contains a fully pinned representative whose invisible
residual is rank-one-free.  Such residuals have a graph normal form and obey
a sharp dimension bound; the literal maximal two-dimensional case is
impossible
(\Cref{thm:terminal-selector,thm:rank-one-free-dimension,thm:division-residual}).

\item \textbf{Separator flux.}
Two-terminal substresses satisfy the load law \eqref{eq:intro-flux}, the skew
dichotomy \(\rho(\xi_a-\xi_b)=0\), and an internal Dirichlet equation.  For a
five-internal-vertex block of rank four, the Pfaffian cofactor zero mode
carries every nonzero flux; internally isolated blocks necessarily use
square-boundary reactions
(\Cref{thm:two-terminal,thm:pfaffian-carrier,thm:separator-budget}).

\item \textbf{An exact five-cycle obstruction.}
Within a precisely defined one-silent induced-five-cycle class, every
one-external completion is impossible, independently of the orientation word
(\Cref{thm:c5-one-external}).
\end{itemize}

These statements do not solve the Heilbronn problem for arbitrary \(n\).
The remaining obstruction is an invisible, rank-one-free shared-kernel
residual.  In the smallest unresolved sector, positive stresses may have
nonzero second-order coefficients of both signs whose weighted sum cancels;
rowwise vanishing is not a consequence of KKT.  In higher dimension, a
two-plane selected from the residual need not inherit the literal silent
vertices used by the low-dimensional classification.  The residual map and
the resulting open questions are stated in \Cref{sec:open}.

The organization follows the mathematics rather than the order in which the
results were found.  Sections~\ref{sec:problem}--\ref{sec:support} establish
the variational and moment theory.  Sections~\ref{sec:shared} and
\ref{sec:terminal} develop observability and the terminal residual.  The
separator laws appear in \Cref{sec:separator}; the exact five-cycle laboratory
and counterexamples follow in Sections~\ref{sec:c5} and \ref{sec:limits}.
Appendices collect normalization checks, the elementary matrix-space lemma,
and the logical description of the exact computation.

\section{The optimization problem and its KKT system}
\label{sec:problem}

\subsection{Area conventions and orientation cells}

Write \(p_i=(x_i,y_i)\) and set
\[
 D_{ijk}(P)
 =\det(p_j-p_i,p_k-p_i),
 \qquad
 A_{ijk}(P)=\frac12|D_{ijk}(P)|.
\]
Thus \(D_{ijk}\) is signed double area and \(A_{ijk}\) is ordinary area.  We
use
\[
 \delta(P)=\min_{i<j<k}A_{ijk}(P)
\]
and reserve \(\Delta\) for the value at the configuration under discussion.
The maximum defining \(\Delta_n\) exists because \([0,1]^{2n}\) is compact
and \(\delta\) is continuous.  Configurations in general position show that
\(\Delta_n>0\) for every fixed \(n\).

At a configuration with \(\delta(P)>0\), every ordered triple has a fixed
orientation in a neighborhood.  If \(s_{ijk}=\operatorname{sgn}D_{ijk}(P)\)
for \(i<j<k\), then on this orientation cell
\[
 A_{ijk}=\frac{s_{ijk}}2D_{ijk}
\]
is a polynomial.  All differential assertions below are made in such a cell.
This entails no loss: a triangle at level \(\Delta>0\) cannot change
orientation without first having zero area.

\begin{proposition}[Empty active triangles]
\label{prop:empty-active}
If \(\delta(P)>0\), every triangle attaining \(\delta(P)\) has no point of
\(P\) in its interior or on one of its sides.
\end{proposition}

\begin{proof}
An interior point divides the triangle into three positive-area triangles
whose areas sum to that of the original triangle.  Each is therefore smaller.
A point on a side creates a collinear triple of area zero.
\end{proof}

The observation is useful when the positive support is interpreted
combinatorially, but it does not by itself control how many minimum triangles
there can be.

\subsection{The smooth max--min formulation}

Introduce a level variable \(\tau\).  Locally the problem is
\begin{equation}
\begin{aligned}
 \text{maximize}\quad &\tau,\\
 \text{subject to}\quad
 & A_t(P)-\tau\ge0 &&(t\in\tbinom{[n]}3),\\
 &x_i\ge0,\quad 1-x_i\ge0,\quad
   y_i\ge0,\quad 1-y_i\ge0 &&(1\le i\le n).
\end{aligned}
\label{eq:maxmin-program}
\end{equation}
Let \(\Tight(P)\) be the set of triangle constraints tight at \(\tau=\Delta\).
Boundary equalities are also called tight; context will distinguish them from
triangle rows.

\begin{lemma}[MFCQ]
\label{lem:mfcq}
At every local maximizer of \eqref{eq:maxmin-program} with \(\Delta>0\), the
Mangasarian--Fromovitz constraint qualification holds.
\end{lemma}

\begin{proof}
Choose a velocity \(v_i\) pointing strictly into the square at every active
side containing \(p_i\); the choices at different vertices are independent,
and a vector pointing into the interior works at a corner.  Let
\[
 M=\max_{t\in\Tight(P)}|D A_t(P)[v]|.
\]
Set the level velocity to \(\dot\tau=-M-1\).  Every active triangle inequality
then has derivative
\[
 D(A_t-\tau)[v,\dot\tau]
 =D A_t[v]+M+1>0,
\]
and every active square inequality has positive derivative by construction.
This is the required strict feasible direction.
\end{proof}

The lemma is included because the multiplier existence used below is not a
genericity assumption.  It is automatic at every positive-area local optimum.
It does not assert strict complementarity.

\subsection{Multipliers and outward reactions}

Let \(\lambda_t\ge0\) be the multiplier of \(A_t-\tau\).  Write
\(\ell_i,r_i,d_i,u_i\ge0\) for the multipliers of
\[
 x_i,\quad 1-x_i,\quad y_i,\quad 1-y_i,
\]
respectively.  Only tight constraints may receive a nonzero multiplier.  With
the maximization convention
\[
 L=\tau+\sum_t\lambda_t(A_t-\tau)
   +\sum_i\{\ell_i x_i+r_i(1-x_i)+d_i y_i+u_i(1-y_i)\},
\]
stationarity in \(\tau\) gives
\begin{equation}
 \sum_t\lambda_t=1.                   \label{eq:lambda-normalization}
\end{equation}
We encode the square multipliers by the outward reaction
\begin{equation}
 b_i=(b_{ix},b_{iy})
      =(r_i-\ell_i,\,u_i-d_i),
 \qquad
 b_i^{\C}=b_{ix}+ib_{iy}.             \label{eq:outward-load}
\end{equation}
For example, a point on the left side has \(b_{ix}=-\ell_i\le0\), while a
point on the right side has \(b_{ix}=r_i\ge0\).  A corner may carry two
nonzero side multipliers.

\begin{theorem}[Normalized KKT multipliers]
\label{thm:kkt}
Every positive-area local optimum admits multipliers
\[
 \lambda_t,\ell_i,r_i,d_i,u_i\ge0
\]
satisfying complementarity, position stationarity, and
\eqref{eq:lambda-normalization}.
\end{theorem}

\begin{proof}
Apply the KKT theorem using \Cref{lem:mfcq}.  The normalization follows from
the level equation as above.
\end{proof}

\begin{table}[t]
\centering
\caption{Core notation.  A subscript \(\mathcal H\) denotes a selected
positive subfamily of active triangles.}
\label{tab:notation}
\begin{tabularx}{\textwidth}{@{}lX@{}}
\toprule
Symbol & Meaning\\
\midrule
\(D_t,A_t\) & signed double area and ordinary area \(A_t=|D_t|/2\)\\
\(\Delta\) & the ordinary minimum area at the configuration\\
\(\lambda_t\) & normalized nonnegative triangle multiplier\\
\(B_t,B\) & elementary triangle stress and its KKT-weighted sum\\
\(b\) & outward square reaction, viewed in \(\R^2\) or \(\C\)\\
\(\KKT(P)\) & normalized KKT multiplier polytope at \(P\)\\
\(\Ksh\) & intersection of kernels over the full KKT face\\
\(H_0\) & quotient \(\Ksh/\langle\one\rangle\)\\
\(W\) & translation-quotiented invisible residual\\
\(Q_i=(\xi_i,\eta_i)\) & a two-coordinate shared shadow\\
\(C_t\) & affine velocity gradient on triangle \(t\)\\
\(q_t\) & second-order area coefficient \(\xi^TB_t\eta\)\\
\(\rho\) & scalar flux across a two-vertex interface\\
\bottomrule
\end{tabularx}
\end{table}

\section{Triangle stresses and affine moment identities}
\label{sec:stress}

\subsection{The elementary skew matrix}

For an oriented active triangle \(t=(i,j,k)\), set
\[
 a_t=e_j-e_i,\qquad c_t=e_k-e_i,
\]
and let \(s_t\in\{1,-1\}\) be its orientation sign.  Define
\begin{equation}
 B_t=s_t(a_tc_t^T-c_ta_t^T).          \label{eq:elementary-stress}
\end{equation}
The matrix is supported on the three vertices of \(t\).  Its nonzero upper
triangular entries are the signed coefficients of the cyclic boundary
\(i\to j\to k\to i\).  Directly from \eqref{eq:elementary-stress},
\begin{equation}
 B_t^T=-B_t,\qquad B_t\one=0.         \label{eq:Bt-basic}
\end{equation}
If \(x=(x_i)\) and \(y=(y_i)\), then
\begin{equation}
 x^TB_ty=s_tD_t(P)=2A_t(P).          \label{eq:stress-area}
\end{equation}
The KKT stress is
\begin{equation}
 B=B(\lambda)=\sum_{t\in\Tight(P)}\lambda_tB_t.  \label{eq:KKT-stress}
\end{equation}
It follows that \(B^T=-B\), \(B\one=0\), and
\[
 x^TBy=2\Delta\sum_t\lambda_t=2\Delta.
\]

\begin{figure}[t]
\centering
\begin{tikzpicture}[scale=0.92,>=Latex,line cap=round,line join=round]
  \begin{scope}[xshift=-3.4cm]
    \coordinate (i) at (0,0);
    \coordinate (j) at (2.6,0.25);
    \coordinate (k) at (0.8,2.1);
    \draw[thick] (i)--(j)--(k)--cycle;
    \draw[->,very thick] ($(i)!0.14!(j)$)--($(i)!0.72!(j)$);
    \draw[->,very thick] ($(j)!0.14!(k)$)--($(j)!0.72!(k)$);
    \draw[->,very thick] ($(k)!0.14!(i)$)--($(k)!0.72!(i)$);
    \fill (i) circle (2pt) node[below left] {\(i\)};
    \fill (j) circle (2pt) node[below right] {\(j\)};
    \fill (k) circle (2pt) node[above] {\(k\)};
    \node at (1.15,0.95) {\(s_t\lambda_t\)};
    \node[below] at (1.25,-0.55) {triangle stress};
  \end{scope}
  \begin{scope}[xshift=2.0cm]
    \draw[thick] (0,0) rectangle (2.6,2.6);
    \fill (0,0.75) circle (2pt);
    \fill (2.6,1.75) circle (2pt);
    \fill (0.95,0) circle (2pt);
    \fill (1.8,2.6) circle (2pt);
    \draw[->,very thick] (0,0.75)--(-0.65,0.75)
      node[left] {\(\ell\)};
    \draw[->,very thick] (2.6,1.75)--(3.25,1.75)
      node[right] {\(r\)};
    \draw[->,very thick] (0.95,0)--(0.95,-0.65)
      node[below] {\(d\)};
    \draw[->,very thick] (1.8,2.6)--(1.8,3.25)
      node[above] {\(u\)};
    \node[below] at (1.3,-0.9) {outward reactions};
  \end{scope}
\end{tikzpicture}
\caption{Left: an oriented triangle contributes a cyclic rank-two skew
stress.  Right: the outward square reactions.  The four side totals are
forced to equal \(\Delta\), not chosen as a normalization.}
\label{fig:stress-boundary}
\end{figure}
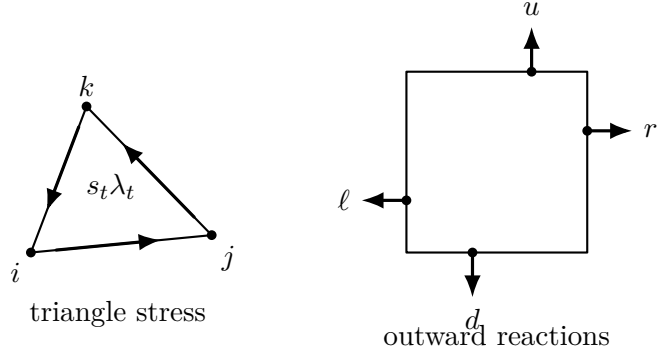

\subsection{Complex equilibrium}

Differentiating \(A_t=\frac12x^TB_ty\) gives
\[
 \nabla_x\sum_t\lambda_tA_t=\frac12By,
 \qquad
 \nabla_y\sum_t\lambda_tA_t=-\frac12Bx.
\]
The position stationarity equations and the convention
\eqref{eq:outward-load} are therefore
\begin{equation}
 By=2b_x,\qquad Bx=-2b_y.             \label{eq:real-equilibrium}
\end{equation}
Upon writing \(z=x+iy\) and \(b=b_x+ib_y\), the two equations become
\begin{equation}
 \boxed{Bz=2ib.}                      \label{eq:complex-equilibrium}
\end{equation}
The complex form is not an additional structure; it is the shortest way of
recording the two planar equilibrium systems with their correct quarter-turn.

\subsection{Full affine moments}

\begin{theorem}[Stress and affine moments]
\label{thm:stress-noether}
Let \(P\) be a positive-area local optimum at ordinary area level \(\Delta\),
and let \((\lambda,b)\) be any normalized KKT pair.  Its stress satisfies
\[
 B^T=-B,\qquad B\one=0,\qquad Bz=2ib,
\]
and the outward reactions obey
\begin{equation}
 \sum_i b_i=0,\qquad
 \sum_i p_i b_i^T=\Delta I_2.         \label{eq:full-noether}
\end{equation}
\end{theorem}

\begin{proof}
Only the two moment identities remain to be shown.  Multiplying
\eqref{eq:real-equilibrium} on the left by \(\one^T\) and using
\(\one^TB=0\) gives force balance.  For the moment tensor, use
\[
 b_x=\frac12By,\qquad b_y=-\frac12Bx.
\]
Thus
\[
 \sum_i p_ib_i^T
 =
 \begin{pmatrix}
  \frac12x^TBy&-\frac12x^TBx\\[2mm]
  \frac12y^TBy&-\frac12y^TBx
 \end{pmatrix}.
\]
The mixed diagonal entries vanish by skew symmetry, while
\(-y^TBx=x^TBy\).  Finally,
\[
 \frac12x^TBy
 =\frac12\sum_t\lambda_tx^TB_ty
 =\sum_t\lambda_t\Delta=\Delta.
\]
\end{proof}

There is also a variational reading.  Translation invariance gives the first
identity in \eqref{eq:full-noether}.  Under \(p_i\mapsto p_i+\varepsilon
Mp_i\), every signed area changes to first order by
\(\operatorname{tr}(M)\); pairing stationarity with arbitrary \(M\) gives the
second identity.  This is the limited sense in which
\eqref{eq:full-noether} is Noether-type.

\subsection{Partial stresses}

The moment identity is not restricted to a whole KKT stress.  Let
\(\mathcal H\subseteq\Tight(P)\) be any selected subfamily, retaining the
weights of a fixed actual KKT multiplier.  Put
\begin{equation}
 B_{\mathcal H}=\sum_{t\in\mathcal H}\lambda_tB_t,
 \qquad
 m_{\mathcal H}=\sum_{t\in\mathcal H}\lambda_t,   \label{eq:partial-stress}
\end{equation}
and define the \emph{algebraic partial load} by
\begin{equation}
 B_{\mathcal H}z=2ib^{\mathcal H}.    \label{eq:partial-load}
\end{equation}
This load need not, by itself, be supported on the square boundary.  At a
vertex shared with other substresses it records the force that those pieces
must cancel.

\begin{theorem}[Partial affine moment identity]
\label{thm:partial-noether}
For every subfamily \(\mathcal H\) as above,
\begin{equation}
 \sum_i b_i^{\mathcal H}=0,\qquad
 \sum_i p_i(b_i^{\mathcal H})^T
   =m_{\mathcal H}\Delta I_2.         \label{eq:partial-noether}
\end{equation}
\end{theorem}

\begin{proof}
The proof of \Cref{thm:stress-noether} applies verbatim to
\(B_{\mathcal H}\).  The only changed scalar is
\[
 \frac12x^TB_{\mathcal H}y
 =\sum_{t\in\mathcal H}\lambda_tA_t
 =m_{\mathcal H}\Delta,
\]
because every \(t\in\mathcal H\) is active at the same level.
\end{proof}

\begin{remark}
Positivity is not needed for the algebra in
\eqref{eq:partial-noether}.  It becomes essential when
\(m_{\mathcal H}>0\) is used geometrically, for instance to rule out a
rank-one moment tensor or to assign a substress to a component.
\end{remark}

\subsection{The four side laws}

Let
\[
 L=\sum_i\ell_i,\quad R=\sum_i r_i,\quad
 D=\sum_i d_i,\quad U=\sum_i u_i.
\]
Force balance gives \(R=L\) and \(U=D\).  The two diagonal entries of
\eqref{eq:full-noether} then give
\begin{equation}
 L=R=D=U=\Delta.                      \label{eq:side-totals}
\end{equation}
Indeed, \(x^Tb_x\) receives the value \(R\) from the right side and zero from
the left side; the vertical identity is analogous.  The off-diagonal entries
give
\begin{equation}
 \sum_i y_ir_i=\sum_i y_i\ell_i,
 \qquad
 \sum_i x_iu_i=\sum_i x_id_i.         \label{eq:side-mixed-moments}
\end{equation}
Thus every positive-area local optimum has a positive reaction on every side
of the square.  The container supplies global information through both the
total reactions and their tangential moments.

\section{Multiplier faces, support, and components}
\label{sec:support}

\subsection{The whole multiplier face}

Fix a positive-area local optimum \(P\), its orientation cell, and its complete
set of tight triangle and square constraints.  The normalized KKT equations
are a finite linear system in the multipliers.  We write
\begin{equation}
 \KKT(P)=\{\theta\ge0:A_P\theta=e\},  \label{eq:KKT-polytope}
\end{equation}
where \(\theta\) collects triangle and boundary multipliers, the last coordinate of
\(e\) encodes \(\sum_t\lambda_t=1\), and the other coordinates encode position
stationarity.  This is a compact polytope: the triangle part lies in a
simplex, and \eqref{eq:side-totals} bounds every boundary multiplier.

Define the \emph{maximal KKT support}
\[
 \mathcal U_+=\{j:\text{some }\theta\in\KKT(P)
                  \text{ has }\theta_j>0\}.
\]
A tight triangle row belonging to \(\mathcal U_+\) is called
\emph{supportable}.  A tight row outside \(\mathcal U_+\) is
\emph{unsupported}.  The latter still belongs to the
tight Jacobian and may be essential for local pinning.

\begin{proposition}[Maximal support and a common exposing vector]
\label{prop:maximal-support}
Every \(\theta^\circ\in\relint\KKT(P)\) has support exactly
\(\mathcal U_+\).  Moreover, there
is a vector \(y\) such that
\begin{equation}
 e^Ty=0,\qquad
 (A_P^Ty)_j=0\ (j\in\mathcal U_+),\qquad
 (A_P^Ty)_j>0\ (j\notin\mathcal U_+). \label{eq:common-exposer}
\end{equation}
\end{proposition}

\begin{proof}
The first assertion is the standard description of the relative interior of
a polytope contained in an orthant: a coordinate is positive at a
relative-interior point exactly when it is not identically zero on the
polytope.  For the second, maximize the sum of the coordinates outside
\(\mathcal U_+\)
over \eqref{eq:KKT-polytope}.  Its value is zero.  Linear programming duality
gives a nonnegative reduced-cost vector that is strictly positive off
\(\mathcal U_+\).  At \(\theta^\circ\), complementary pairing and
\(e^Ty=0\) force every reduced cost on \(\mathcal U_+\) to vanish.  Rescaling
yields \eqref{eq:common-exposer}.
\end{proof}

The vector \(y\) exposes the smallest coordinate face containing the KKT
polytope.  It is a global coefficient-space certificate.  In particular,
\Cref{prop:maximal-support} does not localize \(y\) at a vertex and does not
turn an unsupported tight row into a supportable one.

\begin{corollary}[A sparse KKT certificate]
\label{cor:kkt-sparse}
There is a KKT pair with at most \(2n+1\) positive multipliers in total and at
most \(2n-3\) positive triangle multipliers.
\end{corollary}

\begin{proof}
Choose an extreme point of \(\KKT(P)\).  Its positive columns are linearly
independent, so their number is at most the \(2n+1\) rows of position
stationarity and normalization.  At least one boundary multiplier is positive
on each side by \eqref{eq:side-totals}; removing these four columns leaves at
most \(2n-3\) positive triangle columns.
\end{proof}

This linear support bound is useful, but it says nothing by itself about the
rank of \(B\), the size of a separator block, or reuse of the same external
vertices by several blocks.

\subsection{Positive stress components}

Choose a relative-interior multiplier and form the hypergraph whose vertices
are \([n]\) and whose hyperedges are the triangles with positive multiplier.
A \emph{positive component} is a connected component containing at least one
such hyperedge.  Vertices incident to no positive triangle are omitted from
the component count.  Let the positive components be
\(C_1,\ldots,C_c\), and put
\[
 \alpha_a=\sum_{t\subset C_a}\lambda_t>0.
\]
Because no positive triangle crosses between two components, \(B\) is block
diagonal on their vertex sets, apart from zero rows at vertices incident to no
positive triangle.

\begin{theorem}[Component moments and the two-component bound]
\label{thm:components}
For every positive component \(C_a\),
\begin{equation}
 \sum_{i\in C_a}b_i=0,\qquad
 \sum_{i\in C_a}p_ib_i^T=\alpha_a\Delta I_2.  \label{eq:component-noether}
\end{equation}
Consequently, \(C_a\) carries a positive reaction on each of the four sides of
the square, and
\begin{equation}
 c\le2.                               \label{eq:component-bound}
\end{equation}
\end{theorem}

\begin{proof}
Apply \Cref{thm:partial-noether} to the triangles contained in \(C_a\).  At a
vertex of \(C_a\), no other positive component contributes a stress row.
Hence the algebraic partial load is the full KKT boundary load there, and it
vanishes away from \(C_a\).  This proves \eqref{eq:component-noether}.

The same argument that gave \eqref{eq:side-totals}, now with mass
\(\alpha_a\), shows that the component has side total
\(\alpha_a\Delta>0\) on each of the four sides.  Since every triangle
determined by \(P\) has area at least \(\Delta>0\), no three points are
collinear.  Each side of the square therefore contains at most two points.
There are at most eight point--side incidences, whereas \(c\) components
require at least \(4c\).  Thus \(4c\le8\).
\end{proof}

If \(c=2\), equality holds in the incidence count.  Each component has exactly
one positive contact on each side.  The off-diagonal moments in
\eqref{eq:component-noether} show that the left and right contacts of a
component have a common height, and its bottom and top contacts have a common
horizontal coordinate.  Thus its four reactions form an axial cross:
\[
 (0,v_a),\quad(1,v_a),\quad(u_a,0),\quad(u_a,1).
\]
The following elementary estimate records the geometry of two such crosses.

\begin{proposition}[Two-cross estimate]
\label{prop:two-cross}
If a positive KKT stress has two components, then
\[
 \Delta\le\frac1{16}.
\]
\end{proposition}

\begin{proof}
Let
\[
 a=\min_{j=1,2}\min(u_j,1-u_j),\qquad
 d=\min_{j=1,2}\min(v_j,1-v_j).
\]
Positive minimum area implies \(a,d>0\).  Both heights lie in
\([d,1-d]\), so \(|v_1-v_2|\le1-2d\).  Choose the vertical side nearest to
the bottom--top contact whose distance to a vertical side is \(a\).  The two
component contacts on that vertical side, together with this bottom--top
contact, determine a triangle of area at most
\(a(1-2d)/2\).  Similarly, \(|u_1-u_2|\le1-2a\); using the horizontal side
nearest to the left--right contact that realizes \(d\) gives a triangle of
area at most \(d(1-2a)/2\).  Therefore
\[
 \Delta\le\frac12\min\{a(1-2d),\,d(1-2a)\}.
\]
For completeness, suppose \(a\le d\).  The difference of the two expressions
inside the minimum is \(a-d\), so the minimum is
\(a(1-2d)\le a(1-2a)\le1/8\).  The case \(d\le a\) is symmetric.  Multiplying
by \(1/2\) gives the assertion.
\end{proof}

\begin{remark}
A sharper strict inequality follows when the full inactive-triangle
conditions and the existence of genuine internal positive stresses are used.
We do not need that refinement below.  More importantly,
\eqref{eq:component-bound} does not force the positive support to be connected.
\end{remark}

\section{The shared kernel and hybrid observability}
\label{sec:shared}

\subsection{Why one stress is not intrinsic}

Different points of \(\KKT(P)\) can yield stress matrices with different
kernels.  The degeneracy common to the configuration is therefore
\begin{equation}
 \Ksh(P)=\bigcap_{q\in\KKT(P)}\ker B(q).           \label{eq:shared-kernel}
\end{equation}
Since \(B(q)\one=0\) for every \(q\), constants always belong to \(\Ksh\).
No parity rule for a single skew matrix survives the intersection: even when
each individual kernel has even dimension, the common kernel need not.

\begin{proposition}[Affine-hull description]
\label{prop:shared-affine}
Let \(q^\circ\in\relint\KKT(P)\), let \(T\) be the translation space of the
affine hull of \(\KKT(P)\), and write
\[
 E_d=\sum_t d_tB_t\qquad(d\in T)
\]
for the stress variation.  Then
\[
 \Ksh=\ker B(q^\circ)\cap\bigcap_{d\in T}\ker E_d.
\]
\end{proposition}

\begin{proof}
For \(d\in T\), relative interior gives
\(q^\circ\pm\varepsilon d\in\KKT(P)\) for all sufficiently small
\(\varepsilon>0\).  A vector killed by every actual stress is therefore
killed by the sum and difference of
\(B(q^\circ+\varepsilon d)\) and
\(B(q^\circ-\varepsilon d)\).  This proves one inclusion.  Conversely, every
\(q\in\KKT(P)\) has the form \(q^\circ+d\), and linearity gives the other.
\end{proof}

\begin{theorem}[Finite actual realization]
\label{thm:finite-stack}
There are actual KKT points \(q_1,\ldots,q_s\in\relint\KKT(P)\) such that
\begin{equation}
 \Ksh=\bigcap_{\alpha=1}^s\ker B(q_\alpha).        \label{eq:finite-stack}
\end{equation}
One may take \(s\le n-\dim\Ksh\).
\end{theorem}

\begin{proof}
Fix \(q_*\in\relint\KKT(P)\), and begin with \(V_0=\R^n\).  If
\(V_j\ne\Ksh\), choose \(h\in V_j\setminus\Ksh\).  Some
\(\widetilde q\in\KKT(P)\) satisfies \(B(\widetilde q)h\ne0\).  A sufficiently
small convex combination
\[
 q_{j+1}=(1-\varepsilon)q_*+\varepsilon\widetilde q
\]
lies in the relative interior and can be chosen so that
\(B(q_{j+1})h\ne0\).  Hence
\(V_{j+1}=V_j\cap\ker B(q_{j+1})\) has smaller dimension.  Each step lowers
the dimension by at least one and the process stops at \(\Ksh\).
\end{proof}

This is an exact finite-dimensional statement, not a sampling assertion.
Every matrix in the realizing family is the stress of a normalized
nonnegative KKT multiplier.

\subsection{A canonical Gram operator}

Let
\[
 \mathcal L_B(P)=\Span\{B(q):q\in\KKT(P)\}
 \subset\mathfrak{so}(n).
\]
Give \(\mathfrak{so}(n)\) the Frobenius inner product, choose a
Frobenius-orthonormal basis \(C_1,\ldots,C_s\) of \(\mathcal L_B(P)\), and
define
\begin{equation}
 \Gsh=\sum_{\alpha=1}^sC_\alpha^TC_\alpha.         \label{eq:shared-gram}
\end{equation}

\begin{theorem}[Shared Gram identity]
\label{thm:shared-gram}
The operator \(\Gsh\) is independent of the orthonormal basis,
\(\Gsh\succeq0\), and
\begin{equation}
 \ker\Gsh=\Ksh,\qquad
 \Ksh^\perp=\sum_{C\in\mathcal L_B(P)}\im C.       \label{eq:gram-kernel}
\end{equation}
\end{theorem}

\begin{proof}
If \(D_\beta=\sum_\alpha O_{\beta\alpha}C_\alpha\) is another orthonormal
basis, orthogonality of \(O\) gives
\[
 \sum_\beta D_\beta^TD_\beta
 =\sum_{\alpha,\gamma}
   \left(\sum_\beta O_{\beta\alpha}O_{\beta\gamma}\right)
   C_\alpha^TC_\gamma
 =\sum_\alpha C_\alpha^TC_\alpha.
\]
Moreover,
\[
 h^T\Gsh h=\sum_\alpha\|C_\alpha h\|^2,
\]
so the kernel is the intersection of the kernels of the stress span, which is
\eqref{eq:shared-kernel}.  The image formula is the orthogonal-complement
identity
\((\bigcap\ker C_\alpha)^\perp=\sum\im C_\alpha^T\), together with
\(C_\alpha^T=-C_\alpha\).
\end{proof}

The kernel and stress span are intrinsic.  An unnormalized sum formed from an
arbitrary finite stack has the same kernel when the stack realizes \(\Ksh\),
but generally defines a different metric away from that kernel.  The
orthonormal construction \eqref{eq:shared-gram} removes this ambiguity.

Every \(C\in\mathcal L_B(P)\) has a uniquely determined load
\(\ell(C)\in\C^n\) with
\[
 Cz=2i\ell(C).
\]
Indeed, a linear relation among stress matrices gives the same relation among
loads after evaluation at \(z\).  Thus the actual stacked equilibrium
reconstructs \(z\) modulo the complexification of \(\Ksh\).  We use this only
on the consistent KKT locus; no variational principle on an off-shell stack
is asserted.

\subsection{The tight-row completion}

Let
\[
 V=\R^n\oplus\R^n
\]
be the real velocity space, ordered by the two coordinate vectors.  We use
\(J_{\rm tight}^{\rm tr}\) for the translation-invariant Jacobian of
\emph{all} tight triangle rows and tight square rows.  For boundary contacts,
translation invariance means that we retain the differences of normal
velocities among vertical contacts and, separately, among horizontal
contacts.  The two global translations
\[
 \trans=\Span\{(\one,0),(0,\one)\}
\]
therefore lie in its kernel.  Absolute placement in the square can be restored
after choosing a translation gauge.

Set \(J=J_{\rm tight}^{\rm tr}\), and let
\begin{equation}
 P_{\rm tight}=J^T(JJ^T)^+J             \label{eq:tight-projector}
\end{equation}
be the orthogonal projector onto the tight-row space.  Define
\begin{equation}
 \Ghyb=I_2\otimes\Gsh+P_{\rm tight}.     \label{eq:hybrid-gram}
\end{equation}

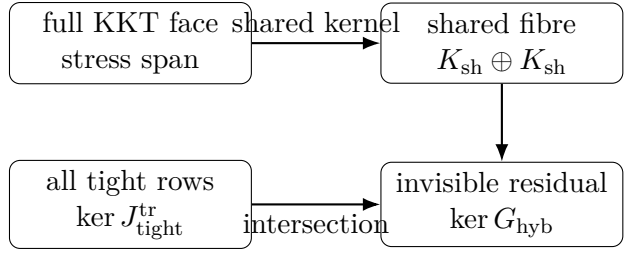
\begin{figure}[t]
\centering
\begin{tikzpicture}[>=Latex,node distance=10mm and 17mm,
  box/.style={draw,rounded corners,align=center,minimum width=3.2cm,
              minimum height=9mm,fill=white}]
  \node[box] (face) {full KKT face\\stress span};
  \node[box,right=of face] (ksh) {shared fibre\\\(\Ksh\oplus\Ksh\)};
  \node[box,below=of face] (tight) {all tight rows\\\(\ker J_{\rm tight}^{\rm tr}\)};
  \node[box,right=of tight] (res) {invisible residual\\\(\ker\Ghyb\)};
  \draw[->,thick] (face)--node[above]{shared kernel}(ksh);
  \draw[->,thick] (tight)--node[below]{intersection}(res);
  \draw[->,thick] (ksh)--(res);
\end{tikzpicture}
\caption{The stress family detects the orthogonal complement of the shared
kernel.  Tight determinant and boundary rows then test the shared fibre.  The
undetected part is their exact intersection.}
\label{fig:hybrid}
\end{figure}

\begin{theorem}[Hybrid kernel identity]
\label{thm:hybrid-kernel}
The hybrid operator is positive semidefinite and
\begin{equation}
 \ker\Ghyb
 =(\Ksh\oplus\Ksh)\cap\ker J_{\rm tight}^{\rm tr}. \label{eq:hybrid-kernel}
\end{equation}
\end{theorem}

\begin{proof}
The projector in \eqref{eq:tight-projector} is positive semidefinite and has
kernel \(\ker J\).  Both summands of \eqref{eq:hybrid-gram} are positive
semidefinite.  A vector has zero quadratic form for their sum precisely when
it has zero quadratic form for each summand.  The first kernel is
\(\Ksh\oplus\Ksh\) by \Cref{thm:shared-gram}, and the second is \(\ker J\).
\end{proof}

\begin{proposition}[Hybrid rank formula]
\label{prop:hybrid-rank}
Let \(\theta_1,\ldots,\theta_s\) be actual KKT multipliers realizing the shared kernel
as in \Cref{thm:finite-stack}, and set
\[
 \mathbb B h=(B(\theta_1)h,\ldots,B(\theta_s)h),
 \qquad k=\dim\Ksh.
\]
Then
\begin{equation}
 \rank\begin{bmatrix}I_2\otimes\mathbb B\\ J_{\rm tight}^{\rm tr}\end{bmatrix}
 =2(n-k)+
   \rank\!\left(J_{\rm tight}^{\rm tr}\big|_{\Ksh\oplus\Ksh}\right).
                                                        \label{eq:hybrid-rank}
\end{equation}
In particular, hybrid observability is equivalent to the left-hand side
having rank \(2n-2\).
\end{proposition}

\begin{proof}
The first block has kernel \(\Ksh\oplus\Ksh\) and rank \(2(n-k)\).
For any linear maps \(A,J\), the rank of the stacked map \((A,J)\) is
\(\rank A+\rank(J|_{\ker A})\): choose a complement of \(\ker A\), whose
image has a nonzero first component, while the image of \(\ker A\) has first
component zero.  Applying this observation proves \eqref{eq:hybrid-rank}.
The final assertion is the rank--nullity form of
\Cref{thm:hybrid-kernel}.
\end{proof}

\begin{definition}[Hybrid observability]
\label{def:hybrid-visible}
The configuration is \emph{hybrid observable} if
\[
 \ker\Ghyb=\trans.
\]
\end{definition}

Let \(\mathcal F\) be a specified \(C^1\) hybrid measurement map on a fixed
smooth orientation and tight-row stratum, with the current hybrid measurement
coordinates, whose derivative has the combined stress and tight-row space
used above.  After fixing two translation gauges, hybrid observability says
that \(D\mathcal F(P)\) is injective.  The constant-rank theorem then gives a
\(C^1\) local inverse from \(\mathcal F\)'s image near \(\mathcal F(P)\) to
the gauge slice near \(P\).  Thus the conclusion is conditional on the
chosen smooth measurement model and is purely local; no off-shell extension
of the KKT stress span, global inverse, or global action is asserted.

\begin{proposition}[An observable branch]
\label{prop:area-rigid-branch}
Suppose the kernel of the tight signed-area Jacobian consists exactly of
infinitesimal special-affine motions
\[
 v_i=Ap_i+c,\qquad A\in\mathfrak{sl}_2(\R),\quad c\in\R^2.
\]
If the translation-invariant boundary rows kill every such motion with
\(A\ne0\), then \(P\) is hybrid observable.
\end{proposition}

\begin{proof}
Every vector in the right-hand side of \eqref{eq:hybrid-kernel} lies in the
kernel of the signed-area Jacobian and of the boundary rows.  The hypotheses
reduce it to a translation.
\end{proof}

The proposition closes the area-rigid, boundary-nondegenerate branch.  The
general existence of a hybrid-observable global optimizer is the principal
open conjecture stated in \Cref{sec:open}.

\section{Terminal residuals and shadow geometry}
\label{sec:terminal}

\subsection{A terminal representative}

Let
\[
 \mathcal M_n=\{P\in[0,1]^{2n}:\delta(P)=\Delta_n\}
\]
be the compact optimal level set.  For a fixed \(P\), the one-point feasible
set at vertex \(i\) is obtained by allowing \(p_i\) and the level to vary while
the other vertices are fixed and all signed triangle and square inequalities
are retained.  It is a polyhedron in \(\R^3\).

\begin{definition}
A configuration is \emph{fully pinned} if \((p_i,\Delta_n)\) is a vertex of
its one-point feasible polyhedron for every \(i\).
\end{definition}

\begin{lemma}[Strictly convex pinning selector]
\label{lem:fully-pinned}
Every optimal level set contains a fully pinned configuration.
\end{lemma}

\begin{proof}
Choose \(P^\Phi\in\mathcal M_n\) maximizing
\[
 \Phi(P)=\sum_i\|p_i\|^2.
\]
If \((p_i^\Phi,\Delta_n)\) were not a vertex of its one-point polyhedron, it
would lie in a nontrivial segment with feasible endpoints.  Neither endpoint
can have level greater than \(\Delta_n\), and replacing only \(p_i^\Phi\) by
either endpoint preserves feasibility at level \(\Delta_n\).  Strict
convexity of \(\|p_i\|^2\) makes \(\Phi\) larger at one endpoint, a
contradiction.
\end{proof}

Put
\[
 H_0=\Ksh/\langle\one\rangle,\qquad r=\dim H_0,
\]
and define the translation-quotiented residual
\begin{equation}
 W(P)=
 \frac{(\Ksh\oplus\Ksh)\cap\ker J_{\rm tight}^{\rm tr}}
      {\trans}
 \subset H_0\otimes\R^2.              \label{eq:terminal-W}
\end{equation}
A tensor \(h\otimes u\) is decomposable, or rank one.  As a velocity it has
the concrete form
\[
 v_i=h_i u:
\]
all vertices move parallel to the same physical direction \(u\), with scalar
speeds \(h_i\).

\begin{lemma}[Exact line motion]
\label{lem:exact-line}
Let \(v_i=h_i u\) represent a nontranslation element of \(W(P)\).  After
adding a translation to \(v\), there is \(\varepsilon>0\) such that
\[
 P+sv\in\mathcal M_n\qquad(|s|<\varepsilon).
\]
Every constraint tight at \(P\) remains exactly tight along this segment.
\end{lemma}

\begin{proof}
The boundary-difference rows imply that all normal velocities at vertical
contacts have a common horizontal component and that all normal velocities at
horizontal contacts have a common vertical component.  Subtract the
corresponding translation.  Every tight square equality is now constant on
the line.

For a triangle \(t=(i,j,k)\), the quadratic term of its determinant is
\[
 \det(v_j-v_i,v_k-v_i)
 =(h_j-h_i)(h_k-h_i)\det(u,u)=0.
\]
The determinant is therefore affine in \(s\).  Its first derivative vanishes
because \(v\in\ker J_{\rm tight}^{\rm tr}\), so every old tight triangle
remains at level \(\Delta_n\).  All other inequalities have positive slack
and remain feasible for small \(|s|\).  The old KKT stresses also remain
valid, since \(Bh=0\), although this last fact is not needed for local
feasibility.
\end{proof}

\begin{theorem}[Terminal rank-one exclusion]
\label{thm:terminal-selector}
There is a fully pinned maximizer \(P^\Phi\in\mathcal M_n\) such that
\begin{equation}
 W(P^\Phi)\cap
 \{h\otimes u:h\in H_0,\ u\in\R^2\}=\{0\}.        \label{eq:rank-one-free}
\end{equation}
\end{theorem}

\begin{proof}
Take the maximizer from \Cref{lem:fully-pinned}.  If a nonzero decomposable
element \(v\) survived, \Cref{lem:exact-line} would give both
\(P^\Phi+sv\) and \(P^\Phi-sv\) in \(\mathcal M_n\) for small \(s\).  But
\[
 \Phi(P^\Phi+sv)+\Phi(P^\Phi-sv)
 =2\Phi(P^\Phi)+2s^2\|v\|^2.
\]
At least one of the two values is larger than \(\Phi(P^\Phi)\), a
contradiction.
\end{proof}

\begin{corollary}[Small shared kernel at the selected representative]
\label{cor:small-shared}
If the representative \(P^\Phi\) selected in
\Cref{thm:terminal-selector} satisfies
\(\dim\Ksh(P^\Phi)\le2\), then it is hybrid observable.
\end{corollary}

\begin{proof}
At this representative \(\dim H_0\le1\).  Every nonzero element of
\(H_0\otimes\R^2\) is decomposable, so \Cref{thm:terminal-selector} gives
\(W=0\).
\end{proof}

\subsection{Rank-one-free matrix spaces}

The residual has a simple linear-algebraic normal form.

\begin{theorem}[Graph normal form and sharp dimension]
\label{thm:rank-one-free-dimension}
Let \(H\) be a real vector space of dimension \(r\), and let
\(W\subset H\otimes\R^2\) contain no nonzero decomposable tensor.  Then
both coordinate projections are injective, and for
\(U=\pi_1(W)\) there is a linear map \(T:U\to H\) such that
\begin{equation}
 W=\{(u,Tu):u\in U\},\qquad
 Tu\notin\R u\quad(u\ne0).            \label{eq:graph-T}
\end{equation}
Moreover,
\begin{equation}
 \dim W\le2\left\lfloor\frac r2\right\rfloor,       \label{eq:sharp-dim}
\end{equation}
and the bound is sharp for every \(r\).
\end{theorem}

\begin{proof}
If \((u,0)\) or \((0,u)\) were a nonzero element of \(W\), it would be
decomposable.  Hence the coordinate projections are injective and give
\eqref{eq:graph-T}.  In particular, \(\dim W\le r\).  If equality holds then
\(U=H\), and \(T\) has no real eigenvector.  A real endomorphism of
odd-dimensional \(H\) has a real eigenvalue, so equality is impossible for
odd \(r\); this yields \eqref{eq:sharp-dim}.  For even \(r\), take
\(W=\{(u,\mathcal Ju):u\in H\}\) for a complex structure \(\mathcal J\).  For odd \(r\), use
the same construction on an even-dimensional hyperplane of \(H\).
\end{proof}

The elementary nature of \Cref{thm:rank-one-free-dimension} is important:
the theorem classifies the size of the obstruction but does not eliminate it.
When \(m=\dim W=r\), \(r\) is even and \(T:H\to H\) has no real invariant
line.  The next result explains what this maximal case forces geometrically.

\subsection{Local affine transfers}
\label{subsec:shadow}

Choose \(v=(\xi,\eta)\in W\) and write
\[
 Q_i=(\xi_i,\eta_i).
\]
The points \(Q_i\) form the \emph{shared shadow} of the velocity.  For a tight
triangle \(t=(i,j,k)\), set
\[
 E_t^P=[\,p_j-p_i\ \ p_k-p_i\,],\qquad
 E_t^Q=[\,Q_j-Q_i\ \ Q_k-Q_i\,],
\]
and define its affine transfer
\begin{equation}
 C_t=E_t^Q(E_t^P)^{-1}.               \label{eq:local-transfer}
\end{equation}
The physical edge matrix is invertible because \(\Delta>0\).  Since
\(E_t^Q=C_tE_t^P\),
\[
 D_t(P+sQ)=D_t(P)\det(I+sC_t).
\]
The two-dimensional identity
\[
 \det(I+sC)=1+s\operatorname{tr}C+s^2\det C
\]
then gives the following formulas.

\begin{proposition}[Trace and shadow determinant]
\label{prop:shadow-transfer}
For every tight triangle and every invisible shadow,
\begin{equation}
 \operatorname{tr}C_t=0,\qquad
 q_t(v):=D^2A_t(P)[v,v]
 =\xi^TB_t\eta
 =2\Delta\det C_t.                    \label{eq:shadow-q}
\end{equation}
\end{proposition}

\begin{proof}
First-order tightness gives the trace equation.  Direct expansion of
\eqref{eq:elementary-stress} gives
\[
 \xi^TB_t\eta
 =s_t\det(Q_j-Q_i,Q_k-Q_i).
\]
The latter determinant equals
\(\det C_t\,D_t(P)\), and \(s_tD_t(P)=2\Delta\).
\end{proof}

\begin{corollary}[Shadow equilibrium and load orthogonality]
\label{cor:shadow-load}
Let \(Q=[\,\xi\ \eta\,]\) be an invisible shadow.  For every actual KKT
stress \(B\) and its load columns \(b_x,b_y\),
\begin{equation}
 BQ=0,
 \qquad
 Q^T[\,b_x\ b_y\,]=0.                \label{eq:shadow-load}
\end{equation}
\end{corollary}

\begin{proof}
Both columns of \(Q\) lie in \(\Ksh\), which gives the first identity.  The
load formulas \(b_x=BY/2\) and \(b_y=-BX/2\), followed by skew symmetry, give
the second one column at a time.
\end{proof}

Thus \(C_t\in\mathfrak{sl}_2(\R)\).  Its determinant distinguishes the
elliptic, hyperbolic, and parabolic local transfer types.  The term
\emph{shadow curvature} below refers only to the second derivative \(q_t\),
not to a differential-geometric curvature.

\begin{lemma}[Transfer jump across a shared edge]
\label{lem:transfer-jump}
Suppose tight triangles \(t,u\) share the physical edge vector \(d\).  For a
fixed invisible shadow, their transfers satisfy
\begin{equation}
 C_u-C_t=a\nu^T,\qquad
 \nu^Td=0,\qquad \nu^Ta=0,\qquad (C_u-C_t)^2=0.  \label{eq:transfer-jump}
\end{equation}
for suitable \(a,\nu\in\R^2\).
\end{lemma}

\begin{proof}
Both transfers send \(d\) to the same shadow edge, so their difference kills
\(d\) and has rank at most one.  Write it as \(a\nu^T\) with
\(\nu^Td=0\).  Both transfers have trace zero, hence
\(0=\operatorname{tr}(a\nu^T)=\nu^Ta\); the square is therefore zero.
\end{proof}

\begin{lemma}[Edge-overlap collinearity]
\label{lem:edge-collinearity}
Suppose two tight triangles share a physical edge and satisfy
\(\operatorname{tr}C_t=q_t=0\) for the same shadow.  Then the four shadow
vertices of their union are collinear.
\end{lemma}

\begin{proof}
Each transfer is trace-free and singular.  If the shared shadow edge is
nonzero, the image of each rank-one transfer is the line spanned by that
edge, so both remaining shadow edges lie on the same line.  If the shared
shadow edge is zero, the physical shared-edge direction lies in both kernels.
A nonzero trace-free rank-one endomorphism is nilpotent and has image equal to
its kernel; both images are therefore parallel to the shared physical edge.
The zero-transfer cases are immediate.
\end{proof}

Consequently, a chain of zero-\(q\) triangles whose shared shadow edges are
nonzero stays on one shadow line.  A collapsed shared shadow edge can switch
lines, so no unqualified edge-connected propagation statement is available.
Blocks meeting only at a vertex may also change line, which is why separator
analysis is necessary.

\subsection{A low-dimensional exclusion}

We record the linear-algebra lemma behind the literal maximal
two-dimensional case.

\begin{lemma}[Dual determinant definiteness]
\label{lem:dual-determinant}
Let \(\mathcal D\subset M_2(\R)\) be two-dimensional and suppose every nonzero
matrix in \(\mathcal D\) is invertible.  Its trace annihilator
\[
 \mathcal D^{\perp_{\rm tr}}
 =\{C:\operatorname{tr}(MC)=0\text{ for every }M\in\mathcal D\}
\]
is two-dimensional, and determinant has a fixed nonzero sign on
\(\mathcal D^{\perp_{\rm tr}}\setminus\{0\}\).
\end{lemma}

\begin{proof}
Right-multiply by the inverse of one element of \(\mathcal D\), so that the
plane becomes \(\Span\{I,A\}\).  Removing the scalar part of \(A\) gives a
trace-free matrix with no real eigenvalue; after real conjugation it is a
nonzero multiple of
\(
\Omega=\left(\begin{smallmatrix}0&-1\\1&0\end{smallmatrix}\right).
\)
The trace annihilator of \(\Span\{I,\Omega\}\) consists of
\[
 \begin{pmatrix}x&y\\y&-x\end{pmatrix},
\]
whose determinant is \(-x^2-y^2\).  Undoing the transformations can reverse
the sign but cannot destroy definiteness.
\end{proof}

Call a vertex \emph{stress-silent} if it belongs to no triangle in the maximal
positive KKT support.  Such a vertex has a zero row and column in every actual
stress and therefore zero KKT reaction.

\begin{theorem}[Exclusion of the literal maximal two-dimensional residual]
\label{thm:division-residual}
At the fully pinned representative of \Cref{thm:terminal-selector}, the
simultaneous conditions
\[
 \dim H_0=2,\qquad \dim W=2
\]
are impossible.
\end{theorem}

\begin{proof}
Choose a basis \(h,g\) of \(H_0\), choose representatives modulo constants,
and put
\[
 R_i=(h_i,g_i)^T.
\]
Writing both scalar coordinates of a residual velocity in this basis
identifies \(W\) with a two-dimensional matrix plane
\(\mathcal D\subset M_2(\R)\): every residual class has a representative
\begin{equation}
 v_i=MR_i,\qquad M\in\mathcal D.       \label{eq:matrix-plane-velocity}
\end{equation}
The rank-one-free property says precisely that every nonzero
\(M\in\mathcal D\) is invertible.

For a tight triangle \(t=(i,j,k)\), define the \emph{base transfer}
\begin{equation}
 C_t^R=[\,R_j-R_i\ \ R_k-R_i\,](E_t^P)^{-1}.       \label{eq:base-transfer}
\end{equation}
The transfer of the velocity \eqref{eq:matrix-plane-velocity} is
\(MC_t^R\).  Since every element of \(W\) is first-order invisible,
\[
 \operatorname{tr}(MC_t^R)=0\qquad(M\in\mathcal D),
\]
so \(C_t^R\in\mathcal D^{\perp_{\rm tr}}\).  Its second coefficient for a
fixed nonzero \(M\) is
\begin{equation}
 q_t(M)=2\Delta\det(MC_t^R)
       =2\Delta\det(M)\det(C_t^R).     \label{eq:division-curvature}
\end{equation}

Let \(\mathcal T_+\) be the maximal positive triangle support, and choose a
relative-interior actual multiplier, so
\(\lambda_t>0\) exactly on \(\mathcal T_+\).  Both scalar coordinates of
\eqref{eq:matrix-plane-velocity} lie in \(\Ksh\), whence
\[
 \sum_{t\in\mathcal T_+}\lambda_tq_t(M)=0.
\]
For fixed \(M\ne0\), \(\det M\ne0\); by
\Cref{lem:dual-determinant}, the nonzero values
\(\det C_t^R\) all have one sign.  Positivity of the weights therefore forces
\[
 C_t^R=0\qquad(t\in\mathcal T_+).     \label{eq:positive-support-collapse}
\]
Thus both \(h\) and \(g\) are constant on every positive component.

Let \(C_1,\ldots,C_c\) be those components, let \(Z\) be the set of
stress-silent labels, and write \(s=|Z|\).  The scalar functions constant on
each \(C_a\) and arbitrary on \(Z\) form a space \(\mathcal S\) of dimension
\(c+s\).  Every actual stress kills \(\mathcal S\), so
\(\mathcal S\subset\Ksh\).  Conversely,
\eqref{eq:positive-support-collapse} and
\(\Ksh=\Span\{\one,h,g\}\) give \(\Ksh\subset\mathcal S\).  Hence
\begin{equation}
 c+s=\dim\Ksh=3.                     \label{eq:atom-count}
\end{equation}
By \Cref{thm:components}, the possibilities are
\((c,s)=(2,1)\) and \((1,2)\).

Let \(V_\partial\) and \(H_\partial\) be the labels on vertical and horizontal
square sides, respectively, and define the contact-difference maps
\[
 \delta_V:H_0\to\R^{V_\partial}/\langle\one\rangle,
 \qquad
 \delta_H:H_0\to\R^{H_\partial}/\langle\one\rangle.
\]
The boundary part of the terminal residual gives
\[
 W\subset\ker\delta_V\oplus\ker\delta_H.
\]
Because \(\dim W=\dim H_0=2\) and both coordinate projections of \(W\) are
injective, they are surjective.  Therefore
\begin{equation}
 \delta_V=0,\qquad\delta_H=0
 \quad\hbox{on }H_0.                  \label{eq:full-boundary-collapse}
\end{equation}
Since \(\Ksh=\mathcal S\) contains every atom-separating scalar function,
all vertical contacts lie in one atom and all horizontal contacts lie in one
atom.  Each positive component has both kinds of contact by
\Cref{thm:components}; hence \(c=2\) is impossible.  We are left with one
positive component \(C\) and two silent labels \(a,b\).  Equation
\eqref{eq:full-boundary-collapse}, together with the contacts of \(C\), also
shows that \(a\) and \(b\) are interior points.

Subtract translations so that every residual velocity vanishes on \(C\).
Evaluation at \(a\) is injective: a nonzero vector in its kernel would be
supported only at \(b\) and would be decomposable.  Since both spaces have
dimension two, evaluation at \(a\), and similarly at \(b\), is an
isomorphism.  Thus every residual velocity has the form
\begin{equation}
 v_a=u,\qquad v_b=Tu,\qquad
 v_j=0\quad(j\in C),                 \label{eq:two-silent-normal-form}
\end{equation}
for some \(T\in GL_2(\R)\).  The map \(T\) has no real eigenvector, since an
eigenvector would make \eqref{eq:two-silent-normal-form} decomposable.  In
particular, \(T-I\) is invertible.

A tight triangle of atom type \(aCC\) or \(bCC\) would have a nonzero
rank-one base transfer in \(\mathcal D^{\perp_{\rm tr}}\), contradicting
\Cref{lem:dual-determinant}.  Thus every tight triangle incident with \(a\)
has type \(abC\).  By \Cref{lem:unique-apex}, at most one such row is tight.
In the one-point linear program at the interior label \(a\), all nonincident
tight rows have normal \((0,0,-1)\) in the variables \((p_a,\tau)\), there is
at most one incident tight row, and there is no boundary normal.  The active
normals have rank at most two, contradicting the rank-three vertex condition
of full pinning.
\end{proof}

\begin{lemma}[Unique-apex calculation]
\label{lem:unique-apex}
In the literal division-plane setting of \Cref{thm:division-residual}, at
most one tight triangle of atom type \(abC\) is incident with a fixed silent
vertex.
\end{lemma}

\begin{proof}
Subtract translations so that every residual velocity vanishes on the
positive component.  Evaluation at either silent label is an isomorphism, so
the velocities can be written
\[
 v_a=u,\qquad v_b=Tu,\qquad v_j=0\quad(j\in C),
\]
where \(T\) has no real eigenvector and hence \(T-I\) is invertible.
Differentiating the oriented area of \((a,b,j)\) for every \(u\) gives
\[
 (T-I)^T\Omega(p_j-p_a)+\Omega(p_b-p_a)=0,
 \qquad
 \Omega=\begin{pmatrix}0&-1\\1&0\end{pmatrix}.
\]
If \(j,j'\in C\) were two apices, subtraction and invertibility of
\((T-I)^T\Omega\) would give \(p_j=p_{j'}\), contradicting positive minimum
area.  The derivation is expanded in \Cref{app:rank-one-free}.
\end{proof}

The literal hypothesis is essential: the atom count and the silent-vertex
argument use the whole two-dimensional quotient \(H_0\), not an arbitrary
plane chosen inside a larger shared kernel.

\begin{theorem}[Silent vertices in a maximal residual]
\label{thm:maximal-silent}
Suppose \(0\ne\dim W=\dim H_0=r\) at a terminal representative.  If \(c\) is the
number of positive components and \(s\) the number of stress-silent vertices,
then
\begin{equation}
 s\ge3-c.                             \label{eq:maximal-silent}
\end{equation}
\end{theorem}

\begin{proof}
In the graph form \eqref{eq:graph-T}, \(T:H_0\to H_0\) has no real
eigenvector.  Choose a nonreal eigenvalue of its complexification.  The real
and imaginary parts of a corresponding eigenvector span a real
two-dimensional \(T\)-invariant subspace
\(U_0=\Span\{u_1,u_2\}\subset H_0\), on which \(T\) has no invariant real
line.

Put \(R_i=((u_1)_i,(u_2)_i)^T\).  For
\(u=\alpha_1u_1+\alpha_2u_2\), invariance of \(U_0\) expresses
\((u,Tu)\) in the form \(v_i=M(\alpha)R_i\).  The matrices
\(M(\alpha)\) form a two-dimensional plane.  A nonzero \(M(\alpha)\) is
singular exactly when \(u\) and \(Tu\) are linearly dependent, which is
impossible.  Thus this plane is a division plane.

For each tight row, use the fixed base transfer formed from the \(R_i\), as in
\eqref{eq:base-transfer}.  First-order invisibility places it in the trace
annihilator of the division plane.  The determinant-definiteness and
positive-weight cancellation argument leading to
\eqref{eq:positive-support-collapse} then shows that both \(u_1\) and \(u_2\)
are constant on every positive component.

Modulo constants, the scalar functions constant on \(c\) positive components
and arbitrary on \(s\) silent vertices form a space of dimension \(c+s-1\).
The two-dimensional space \(U_0\) injects into this quotient, so
\(2\le c+s-1\), which is the claimed inequality.
\end{proof}

Unlike \Cref{thm:division-residual}, this theorem is not an exclusion.  The
number of silent vertices is not bounded above in a higher-dimensional shared
kernel.

\subsection{Flat and mixed positive support}
\label{subsec:flat-mixed}

For every actual stress and every invisible \(v=(\xi,\eta)\),
\begin{equation}
 \sum_t\lambda_tq_t(v)=\xi^TB\eta=0.  \label{eq:weighted-q}
\end{equation}
There are two fundamentally different possibilities.

\begin{description}[leftmargin=2.4cm,style=nextline]
\item[Flat support.]
\(q_t(v)=0\) for every triangle in the maximal positive support.  The
edge-overlap lemma then organizes chains with noncollapsed shared shadow
edges into shadow lines.

\item[Mixed support.]
Some positive rows have nonzero \(q_t\), necessarily with cancellations in
\eqref{eq:weighted-q}.  Since determinant on
\(\mathfrak{sl}_2(\R)\) is indefinite, this is compatible with all
first-order equations.
\end{description}

The implication
\[
 \sum_t\lambda_tq_t=0
 \quad\Longrightarrow\quad
 q_t=0\ \text{for every positive row}
\]
is false without an additional definite-plane hypothesis such as the one in
\Cref{thm:division-residual}.  The mixed sector remains one of the explicit
open branches.

\section{Partial stresses across a two-vertex separator}
\label{sec:separator}

\subsection{The two-terminal law}

Let \(\mathcal H\) be a positive subfamily of active triangles.  Its vertices
are partitioned as
\[
 V(\mathcal H)=I\sqcup\{a,b\},
\]
where \(I\) is the internal set and \(a,b\) are terminals.  For the first
three results, no isolation from the rest of the positive support is needed.
Let \(\xi\in\Ksh\), and suppose the partial residual vanishes at every
internal vertex:
\begin{equation}
 (B_{\mathcal H}\xi)_i=0\qquad(i\in I).            \label{eq:internal-vanish}
\end{equation}

\begin{theorem}[Two-terminal residual and load law]
\label{thm:two-terminal}
Under \eqref{eq:internal-vanish}, there is a scalar \(\rho\) such that
\begin{equation}
 B_{\mathcal H}\xi=\rho(e_a-e_b).                  \label{eq:terminal-residual}
\end{equation}
If \(p_i=(X_i,Y_i)\), then
\begin{equation}
 \sum_i\xi_i(b_i^{\mathcal H})^T
 =\frac{\rho}{2}(Y_b-Y_a,\,X_a-X_b)
 =\frac{\rho}{2}\Omega(p_a-p_b)^T,                \label{eq:terminal-load}
\end{equation}
where \(\Omega(x,y)=(-y,x)\).  In addition,
\begin{equation}
 \rho(\xi_a-\xi_b)=0.                              \label{eq:skew-dichotomy}
\end{equation}
\end{theorem}

\begin{proof}
The vector \(B_{\mathcal H}\xi\) is supported on \(a,b\).  Its coordinates
sum to zero because \(\one^TB_{\mathcal H}=0\), which gives
\eqref{eq:terminal-residual}.

From \eqref{eq:partial-load},
\[
 b_x^{\mathcal H}=\frac12B_{\mathcal H}Y,\qquad
 b_y^{\mathcal H}=-\frac12B_{\mathcal H}X.
\]
Skew symmetry now gives
\[
\begin{aligned}
 \xi^Tb_x^{\mathcal H}
 &=-\frac12(B_{\mathcal H}\xi)^TY
   =\frac{\rho}{2}(Y_b-Y_a),\\
 \xi^Tb_y^{\mathcal H}
 &=\frac12(B_{\mathcal H}\xi)^TX
   =\frac{\rho}{2}(X_a-X_b).
\end{aligned}
\]
Finally,
\[
 0=\xi^TB_{\mathcal H}\xi=\rho(\xi_a-\xi_b).
\]
\end{proof}

The last identity splits the analysis at no cost: either the flux vanishes or
the shared coordinate collapses across the terminals.  If \(\rho\ne0\), put
\[
 c=\xi_a=\xi_b.
\]
Combining force balance with \eqref{eq:terminal-load} then gives the internal
form
\begin{equation}
 \sum_{i\in I}(\xi_i-c)(b_i^{\mathcal H})^T
 =\frac{\rho}{2}\Omega(p_a-p_b)^T.                \label{eq:internal-load-law}
\end{equation}

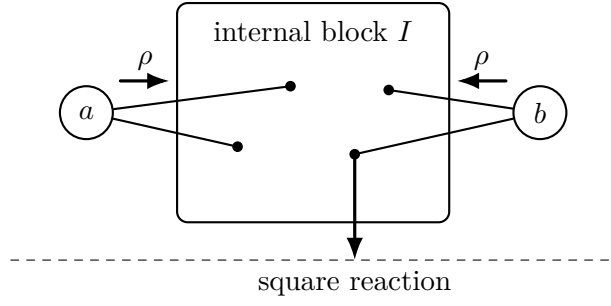
\begin{figure}[t]
\centering
\begin{tikzpicture}[>=Latex,line cap=round,line join=round]
  \node[circle,draw,thick,minimum size=7mm] (a) at (-3,0) {\(a\)};
  \node[circle,draw,thick,minimum size=7mm] (b) at (3,0) {\(b\)};
  \draw[rounded corners,thick] (-1.8,-1.45) rectangle (1.8,1.45);
  \node at (0,1.08) {internal block \(I\)};
  \foreach \x/\y in {-1/-.45,-.3/.35,.55/-.55,1/.3}
    \fill (\x,\y) circle (2pt);
  \draw[thick] (a)--(-1,-.45);
  \draw[thick] (a)--(-.3,.35);
  \draw[thick] (b)--(.55,-.55);
  \draw[thick] (b)--(1,.3);
  \draw[->,very thick] (-2.55,.42)--(-1.9,.42)
    node[midway,above] {\(\rho\)};
  \draw[->,very thick] (2.55,.42)--(1.9,.42)
    node[midway,above] {\(\rho\)};
  \draw[->,very thick] (.55,-.55)--(.55,-1.95)
    node[below] {square reaction};
  \draw[dashed] (-4,-1.95)--(4,-1.95);
\end{tikzpicture}
\caption{A positive substress with terminals \(a,b\).  Internal vanishing
leaves the antisymmetric terminal residual \(\rho(e_a-e_b)\).  If the block is
isolated, the partial moment identity forces a nonzero square reaction at an
internal vertex.}
\label{fig:separator}
\end{figure}

\subsection{The Dirichlet carrier}

Assume \(\rho\ne0\), set \(w=\xi-c\one\), and order the vertices with the two
terminals first.  Write
\begin{equation}
 B_{\mathcal H}=
 \begin{pmatrix}
  A&E\\
  -E^T&C
 \end{pmatrix},                                   \label{eq:dirichlet-block}
\end{equation}
where \(C\) is the internal skew block.  Since \(w_a=w_b=0\),
\eqref{eq:terminal-residual} is equivalent to
\begin{equation}
 Cw_I=0,\qquad Ew_I=\rho(1,-1)^T.                 \label{eq:dirichlet-equations}
\end{equation}

\begin{proposition}[Dirichlet flux carrier]
\label{prop:dirichlet-carrier}
A nonzero two-terminal flux requires a zero mode of the internal skew block
whose boundary trace under \(E\) is nonzero.  In particular, if \(C\) is
invertible, then \(\rho=0\).
\end{proposition}

\begin{proof}
This is immediate from \eqref{eq:dirichlet-equations}.
\end{proof}

The proposition explains a parity trap in a natural elimination.  If the
internal set has five vertices, \(C\) is a \(5\times5\) skew matrix and is
automatically singular.  The Schur complement that would eliminate the
internal variables is unavailable precisely in the potentially nonzero-flux
chart.

\begin{theorem}[Pfaffian carrier on five internal vertices]
\label{thm:pfaffian-carrier}
Let \(|I|=5\).  Define
\begin{equation}
 \kappa_i=(-1)^{i-1}\Pf(C_{\widehat i}),
 \qquad i=1,\ldots,5,                              \label{eq:pfaffian-vector}
\end{equation}
where \(C_{\widehat i}\) is obtained by deleting row and column \(i\).  Then
\[
 C\kappa=0.
\]
If \(\rank C=4\), then \(\ker C=\langle\kappa\rangle\), and every nonzero
flux has the form
\begin{equation}
 w_I=t\kappa,\qquad
 \rho(1,-1)^T=tE\kappa.                            \label{eq:pfaffian-port}
\end{equation}
In triangle-stress coordinates the resulting port flux is a homogeneous
cubic polynomial in the signed row weights.
\end{theorem}

\begin{proof}
The identity \(C\kappa=0\) is the Pfaffian cofactor identity for an odd skew
matrix.  When \(\rank C=4\), its kernel is one-dimensional, so
\eqref{eq:dirichlet-equations} gives \eqref{eq:pfaffian-port}.  Each component
of \(\kappa\) is quadratic in the entries of \(C\), while \(E\) is linear in
the triangle weights.
\end{proof}

For a concrete row convention, let \(\alpha_{ij}\), \(\beta_{ij}\), and
\(\gamma_i\) be the signed weights of the rows
\((a,i,j)\), \((b,i,j)\), and \((a,b,i)\), respectively.  Then
\[
 C_{ij}=\alpha_{ij}+\beta_{ij}
\]
and, up to the global sign chosen for \(\rho\), the cubic carrier is
\begin{equation}
 \mathcal P_{\rm flux}=
 \sum_{i<j}\alpha_{ij}(\kappa_i-\kappa_j)
 -\sum_i\gamma_i\kappa_i.                         \label{eq:cubic-flux}
\end{equation}
The formula is an obstruction polynomial, not an assertion that every
coefficient choice is geometrically realizable.

\subsection{Boundary reactions and block budgets}

\begin{definition}
The positive block \(\mathcal H\) is \emph{internally isolated} if no positive
triangle outside \(\mathcal H\) contains a vertex of \(I\).  Blocks are
\emph{pairwise internally disjoint} if their internal vertex sets are
disjoint.
\end{definition}

\begin{theorem}[Internal reaction]
\label{thm:internal-reaction}
An internally isolated positive two-terminal block has a nonzero full KKT
reaction at an internal vertex.
\end{theorem}

\begin{proof}
Suppose \(b_i^{\mathcal H}=0\) for every \(i\in I\).  Partial force balance
leaves opposite loads at \(a,b\).  The moment tensor
\[
 \sum_i p_i(b_i^{\mathcal H})^T
\]
then has rank at most one, being an outer product of \(p_a-p_b\) with one
load vector.  By \Cref{thm:partial-noether}, it equals
\(m_{\mathcal H}\Delta I_2\), which has rank two because
\(m_{\mathcal H},\Delta>0\).  This is impossible.  Isolation means that the
nonzero partial load at an internal vertex is not cancelled by another
positive block, so it is the full boundary reaction.
\end{proof}

\begin{theorem}[Separator budgets]
\label{thm:separator-budget}
There are at most eight pairwise internally disjoint, internally isolated
positive two-terminal blocks.  For a fixed terminal pair and a fixed shared
coordinate \(\xi\), at most four such pairwise internally disjoint blocks can
have nonzero flux.
\end{theorem}

\begin{proof}
By \Cref{thm:internal-reaction}, each block contains a distinct internal point
with a square reaction.  No side contains three points when \(\Delta>0\), so
the four sides contain at most eight distinct reaction points.  This proves
the first bound.

For the second, nonzero flux gives \(\xi_a=\xi_b\).  Put
\(R_i=(\xi_i,Y_i)\).  If \(R_a\ne R_b\), then \(Y_a\ne Y_b\), and the first
component of \eqref{eq:internal-load-law} is nonzero.  Some internal reaction
therefore has a horizontal component and lies on one of the two vertical
sides.  If \(R_a=R_b\), the physical points \(p_a,p_b\) are distinct, so
\(X_a\ne X_b\); the second component is nonzero and some internal reaction
lies on a horizontal side.  For the fixed terminal pair the relevant pair of
parallel sides is the same for every block.  Those two sides contain at most
four points.
\end{proof}

\begin{remark}
The qualifiers in \Cref{thm:separator-budget} are substantive.  The theorem
does not bound the number of vertices or positive rows inside one block, does
not bound recursive separator depth or treewidth, and does not apply to
overlapping blocks whose partial loads may cancel at a shared internal
vertex.
\end{remark}

\section{An exact one-silent five-cycle laboratory}
\label{sec:c5}

The structural theory above is valid for every \(n\).  We now study one
literal residual class in which the general equations reduce to a finite
symbolic problem.  The result is deliberately stated with all of its
hypotheses.

\subsection{The class}

Assume
\[
 \dim H_0=2,\qquad \dim W=1,
\]
and choose affine shadow coordinates in the form
\begin{equation}
 Q_0=(0,1),\qquad
 Q_i=(\xi_i,0)\quad(1\le i\le5).      \label{eq:one-silent-normal-form}
\end{equation}
The label \(0\) is the unique stress-silent vertex.  Indices
\(i\in\{1,\ldots,5\}\) are read cyclically.  Suppose:
\begin{enumerate}[label=(C\arabic*),leftmargin=2.4em]
\item the five rows
\[
 (0,1,2),(0,2,3),(0,3,4),(0,4,5),(0,5,1)
\]
are tight but unsupported;
\item every other triangle on the old labels \(0,\ldots,5\) is strict;
\item \(\xi_{i+1}-\xi_i\ne0\) for every cycle edge;
\item each core label \(1,\ldots,5\) is nonsilent;
\item first-order trace equations hold on every tight row, and the shadow
coordinates lie in the shared kernel of every actual stress;
\item the relevant positive component covers all five core labels, uses
exactly one further label \(e\), and has a relative-interior multiplier
positive on its maximal support; the same shadow normalization satisfies
\(Q_e=(\xi_e,0)\);
\item the full KKT equilibrium, partial moment identities, square boundary
conditions, and inactive inequalities hold.
\end{enumerate}
The five old rows form a strict induced \(C_5\).  They do not receive positive
KKT multiplier: each contains the silent label \(0\).  Their trace equations
are first-order equations and must not be confused with the second-order
flat-support condition of \Cref{subsec:flat-mixed}.

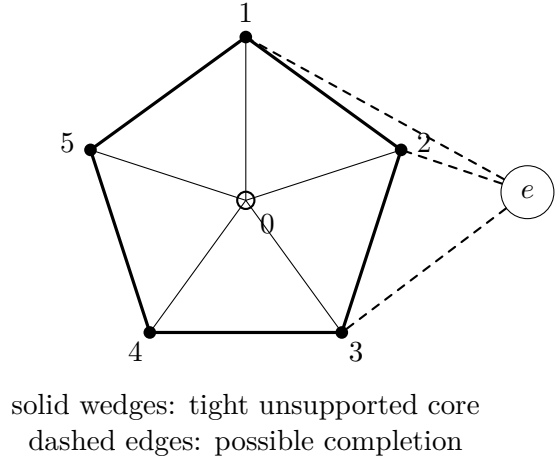
\begin{figure}[t]
\centering
\begin{tikzpicture}[>=Latex,scale=1.08,line cap=round,line join=round]
  \foreach \i/\a in {1/90,2/18,3/-54,4/-126,5/162}{
    \coordinate (v\i) at (\a:2.0);
    \fill (v\i) circle (2.2pt) node[shift={(\a:0.32)}] {\(\i\)};
  }
  \coordinate (o) at (0,0);
  \fill[white] (o) circle (3pt);
  \draw[thick] (o) circle (3pt) node[below right=2pt] {\(0\)};
  \draw[very thick] (v1)--(v2)--(v3)--(v4)--(v5)--cycle;
  \foreach \i in {1,...,5}{\draw[thin] (o)--(v\i);}
  \node[draw,circle,minimum size=7mm] (e) at (3.45,0.1) {\(e\)};
  \draw[dashed,thick] (e)--(v1);
  \draw[dashed,thick] (e)--(v2);
  \draw[dashed,thick] (e)--(v3);
  \node[align=center] at (0,-2.75)
    {solid wedges: tight unsupported core\\dashed edges: possible completion};
\end{tikzpicture}
\caption{The one-silent \(C_5\) class.  The five old tight triangles use the
silent vertex \(0\) and the cycle edges; they are unsupported.  A
one-external completion may use only rows \((e,i,j)\).  The drawing is
combinatorial, not a plot of the exact coordinates.}
\label{fig:c5}
\end{figure}

\subsection{A strict local realization}

The class is not empty at the determinant-and-trace level.  There is an exact
six-label realization over \(\mathbb Q(\sqrt{1749})\) with orientation word
\[
 (+,+,-,+,-).
\]
After normalizing the active signed double areas to have modulus \(1\), all
fifteen inactive old triangles have strictly larger modulus.  The smallest is
\begin{equation}
 \frac{8003-93\sqrt{1749}}{3180}
 =1.2935983012\ldots .                \label{eq:c5-slack}
\end{equation}
The silent physical point is strictly inside a triangle of core points:
\[
 p_0=\frac{15}{61}p_1+\frac{31}{61}p_2+\frac{15}{61}p_3.
\]
All five adjacent shadow differences and all five shadow determinants
\(q_t\) are nonzero.

\begin{proposition}[Smooth local survivor]
\label{prop:c5-survivor}
After fixing the natural affine gauge, the five equal-area and five trace
equations defining the exact \(C_5\) realization have Jacobian rank \(10\) in
\(14\) variables.  The realization therefore lies on a smooth real local
component of dimension \(4\).  The strict inactive inequalities persist on a
neighborhood of that component.
\end{proposition}

\begin{proof}
The exact Jacobian has a nonzero rational-algebraic \(10\times10\) minor.
The implicit function theorem gives the dimension statement, and
\eqref{eq:c5-slack} gives persistence of the inequalities.  The minor and all
fifteen exact slacks are reproduced in the supplement.
\end{proof}

Thus the local residual is not killed by a finite-order singularity.  What
fails, as proved next, is completion with only one additional
positive-support label.

\subsection{The one-external theorem}

\begin{lemma}[Exhaustive collapsed-fan reduction]
\label{lem:c5-fan-reduction}
Under {\rm(C1)--(C7)}, let \(G_e\) be the graph on the five core labels in
which \(ij\) is an edge when the row \((e,i,j)\) has positive multiplier.
Put
\[
 R_i=(\xi_i,Y_i),\qquad
 u_i=\xi_i-\xi_e,\qquad v_i=Y_i-Y_e,\qquad
 S=\{i:u_i=0\}.
\]
Then \(S\) is a nonempty independent set of the old five-cycle,
\(|S|\le2\), and \(G_e[S^c]\) has no edge.  Every remaining support has one
of the following forms.
\begin{enumerate}[label=(F\arabic*),leftmargin=2.4em]
\item There is one \(R\)-collapsed centre \(c\in S\), joined to all four
other core labels, with \(R_c=R_e\), and for one \(h>0\),
\[
 Y_j=Y_e+\varepsilon_jh,\qquad \varepsilon_j\in\{-1,1\}.
\]
\item \(S=\{c,d\}\) with \(c,d\) nonadjacent in the old cycle.  If the
support is not already of type {\rm(F1)}, both centres have outside
neighbours, are collapsed to \(R_e\), and the three outside labels have
heights \(Y_e\pm h\) for one common \(h>0\).
\end{enumerate}
\end{lemma}

\begin{proof}
Strict inducedness excludes positive rows using only old nonsilent labels,
and silence excludes positive rows containing \(0\).  Thus every positive row
in the relevant component has the form \((e,i,j)\), and nonsilence says that
\(G_e\) has no isolated vertex.  The trace equation for that row is
\begin{equation}
 \det(R_i-R_e,R_j-R_e)=0.             \label{eq:c5-R-collinearity}
\end{equation}

Let \(H\) be a connected component of \(G_e[S^c]\) containing an edge, and
let \(B_H\) be the partial stress formed by its internal edges.  At
\(i\in H\), a row joining \(i\) to \(s\in S\) contributes a signed positive
multiple of \(u_s=0\) to \((B\xi)_i\).  Every other contribution is internal
to \(H\).  Hence the full equation \(B\xi=0\) gives
\((B_H\xi)_i=0\) on \(H\); row-sum zero also kills the only remaining
coordinate, at \(e\).  Therefore
\begin{equation}
 B_H\xi=0.                            \label{eq:c5-partial-kernel}
\end{equation}

All vectors \(R_i-R_e\), \(i\in H\), are nonzero.  Equation
\eqref{eq:c5-R-collinearity} propagates one line through \(R_e\) across the
connected graph \(H\).  Write this line as
\(\alpha\xi+\beta Y+\gamma=0\).  If \(\alpha=0\), then \(Y\) is constant on
the partial support.  Force balance makes
\(\sum_iY_i(b_i^H)^T=0\), contradicting the second row
\((0,m_H\Delta)\) of the partial affine moment.  If \(\alpha\ne0\), write
\(\xi=aY+c\).  From \eqref{eq:c5-partial-kernel}, skew symmetry, force
balance, and the same moment identity,
\[
 0=\sum_i\xi_i(b_i^H)^T
  =a(0,m_H\Delta).
\]
Thus \(a=0\), so \(\xi_i=\xi_e\) on \(H\), contrary to \(H\subset S^c\).
It follows that \(G_e[S^c]\) has no edge.

If two adjacent old-cycle labels lay in \(S\), their shadow difference would
vanish, contradicting {\rm(C3)}.  Hence \(S\) is independent and
\(|S|\le2\).  It is nonempty because \(G_e[S^c]\) has no edge while \(G_e\)
has no isolated vertex.

If \(S=\{c\}\), absence of edges outside \(S\) and coverage force \(c\) to be
joined to all four other labels.  For any neighbour \(j\),
\[
 0=\det((0,v_c),(u_j,v_j))=-v_cu_j.
\]
Since \(u_j\ne0\), \(v_c=0\), so \(R_c=R_e\).  The positive rows
\((e,c,j)\) all have signed double-area modulus \(D=2\Delta\); hence
\[
 |X_c-X_e|\,|Y_j-Y_e|=D.
\]
The first factor is fixed and nonzero, giving the common height modulus \(h\)
in {\rm(F1)}.

Now let \(S=\{c,d\}\).  The two centres are nonadjacent in the old cycle.  At
a centre, one outside neighbour cannot satisfy its scalar kernel equation:
the edge \(cd\), if present, contributes zero because \(u_c=u_d=0\), while a
single outside edge contributes a nonzero multiplier times a nonzero \(u_j\).
Thus a centre with outside neighbours has at least two.  If both centres have
outside neighbours, their neighbour sets among three outside labels overlap.
Equation \eqref{eq:c5-R-collinearity} collapses both centres to \(R_e\), and
a common neighbour equates their two height moduli.  Coverage gives {\rm(F2)}.
If only one centre has outside neighbours, it is joined to all three outside
labels, while the other centre is covered by \(cd\).  The equal-area equation
for \((e,c,d)\) gives the same height modulus, so this is already an
{\rm(F1)} fan with four leaves.  The cases are exhaustive.
\end{proof}

\begin{theorem}[All-sign one-external exclusion]
\label{thm:c5-one-external}
Under conditions {\rm(C1)--(C7)}, no positive KKT completion using only the
single external label \(e\) exists, for any of the \(32\) orientation words
of the five old rows.
\end{theorem}

\begin{proof}
We separate the conceptual reduction from the finite exact verification.
By \Cref{lem:c5-fan-reduction}, the two displayed fan families exhaust all
supports before any orientation sign is used.

It remains to solve the area and trace equations for those types.  The action
of the dihedral group \(D_5\), together with simultaneous reversal of all
orientations, has four word orbits, represented by
\[
 +++++,\qquad ++++-,\qquad +++--,\qquad ++-+-.
\]
For these representatives the exact elimination checks \(480\) generic fan
patterns.  Denominator-zero cases are not discarded: \(1360\) singular charts
are solved in ungauged coordinates by exact rank and polynomial reduction.
An independent exceptional-rank implementation checks a further \(1440\)
cases.  No chart yields a feasible survivor; intermediate real candidates
are tested exactly against the required inactive inequalities.  The reduction
and verification architecture are described in \Cref{app:computation}.
\end{proof}

\begin{definition}
The external completion depth \(d_{\rm ext}(C_5)\) is the minimum number of
labels in a finite set \(E\), disjoint from \(\{0,\ldots,5\}\), used by a
positive component that covers all five nonsilent core labels and satisfies
{\rm(C1)--(C5)} and {\rm(C7)}, with {\rm(C6)} replaced by: the component has a
relative-interior positive multiplier, uses exactly the external label set
\(E\), and satisfies \(Q_e=(\xi_e,0)\) for every \(e\in E\).  If no such
completion exists, set \(d_{\rm ext}(C_5)=+\infty\).
\end{definition}

\begin{corollary}
\label{cor:external-depth}
In the literal class above,
\[
 d_{\rm ext}(C_5)\ge2.
\]
\end{corollary}

The corollary does not say that two external labels are impossible.  It also
does not show that every one-silent residual contains a five-cycle, that a
non-induced five-cycle reduces to this class, or that the class is stable
under selecting a two-plane from a higher-dimensional residual.

\subsection{Same-label endpoints}

The exact local realization provides another useful finite statement.  There
are fifteen old strict triangle rows, and either orientation could become
tight at a same-six-label endpoint.

\begin{proposition}[No same-six endpoint]
\label{prop:same-six}
For the exact strict realization with word \(++-+-\), none of the \(30\)
signed choices of a new old-label trace-closed row gives a feasible endpoint.
\end{proposition}

\begin{proof}
Exact elimination gives the unit ideal in \(25\) cases.  Two cases have only
nonreal roots.  The remaining three have real algebraic solutions, but each
violates an inactive inequality: two force another triangle to have area
\(1/20\) of the normalized double-area scale, and one creates a collinear
triple.  The case table is included in the supplement.
\end{proof}

Together, \Cref{prop:c5-survivor,prop:same-six,thm:c5-one-external} give a
clean local-to-global picture: the six-label germ is genuinely smooth, it
cannot enlarge using the same labels, and one new positive-support label is
still insufficient.

\section{Limits, counterexamples, and finite checks}
\label{sec:limits}

The structural statements above depend on three features that are easy to
conflate: positive support rather than mere tightness, the whole KKT face
rather than one multiplier, and determinant geometry rather than generic
smooth optimization.  We record exact examples that separate these features.
They are not pathological exceptions to the theory; they mark its logical
boundary.

\subsection{A tight row need not be supported}

\begin{example}[A generic selector counterexample]
\label{ex:selector-countermodel}
On
\(
 D=[-1,1]\times[-\tfrac12,\tfrac12]
\)
consider the max--min problem for
\[
 f_1=x,\qquad f_2=-x,\qquad
 f_3=y+y^2,\qquad f_4=-y^4.
\]
Its unique maximizer is \((0,0)\), where all four functions vanish.  The
normalized KKT face is
\[
 \lambda_1=\lambda_2,\qquad \lambda_3=0,
 \qquad \lambda_4=1-2\lambda_1,
 \qquad 0\le\lambda_1\le\tfrac12.
\]
Thus the third row is tight at every optimum but unsupported throughout the
whole multiplier face.  In the vertical direction its second derivative is
\(2\), while every supported row has second derivative zero.
\end{example}

\begin{proof}
For every \((x,y)\in D\),
\(
 \min(f_1,f_2)=-|x|\le0
\)
and \(f_4=-y^4\le0\).  Equality of the minimum to zero forces \(x=y=0\),
proving uniqueness.  The displayed face follows by differentiating at the
origin and imposing nonnegativity and normalization.  The Hessian statement
is immediate.
\end{proof}

This is a two-variable smooth optimization example, not a realization by
triangle determinants in a square.  It shows why no selection argument based
only on compactness, uniqueness, or relative-interior multipliers can promote
all tight rows into positive support.

\subsection{Whole-face cancellation need not be rowwise}

\begin{example}[A midpoint flux counterexample]
\label{ex:midpoint-flux}
Let
\[
 p_6=(0,\tfrac12),\quad p_7=(1,\tfrac12),\quad
 p_1=(\tfrac12,1),\quad p_2=(\tfrac12,0),
 \qquad \xi=(0,0,1,1),
\]
where the entries of \(\xi\) are ordered as \((6,7,1,2)\).  The four
oriented rows
\[
 (6,7,1),\qquad (6,2,7),\qquad
 (6,2,1),\qquad (7,1,2)
\]
have ordinary area \(1/4\).  The sum of the first pair and the sum of the
second pair generate two positive KKT rays with the same square-normal load;
both annihilate \(\xi\).  After normalization their convex hull is a
one-dimensional multiplier face.  Yet the first two single-row blocks have
opposite nonzero terminal fluxes, \(-1\) and \(+1\).
\end{example}

The one-point pinned ranks are \((3,3,3,3)\), and no triangle inequality is
inactive at this four-point level.  The rowwise shadow traces are
\((0,0,1,-1)\).  Hence this example does not meet the literal five-cycle
hypotheses of \Cref{thm:c5-one-external}; it demonstrates instead that a
whole-face equation such as \(B\xi=0\) does not force the corresponding
partial fluxes, or these rowwise trace quantities, to vanish separately.  The
coordinates, loads, ranks, and two extreme multipliers are checked exactly in
the supplement.

\subsection{Why finite-order lifting is insufficient}

The determinant of a moving triangle is quadratic in the motion parameter.
Consequently, first- and second-order equalities completely control one
fixed affine line, but they do not choose a globally feasible direction from
a multidimensional residual.  Nor does a formal solution of a finite jet
system ensure compatibility with inactive triangle inequalities, square
contacts, or a positive KKT multiplier.  The strict local survivor in
\Cref{prop:c5-survivor} makes this distinction concrete: its area--trace
germ is smooth and has positive-dimensional freedom, while
\Cref{thm:c5-one-external} rules out a specified global completion.

Likewise, a graph or oriented-matroid description retains incidences and
signs but discards the metric moment tensor
\(
 \sum_i p_i b_i^T=\Delta I_2
\)
and the local transfer determinants.  Such combinatorial data are useful for
organizing cases, but they cannot by themselves certify the analytic
reaction and flux constraints used here.

\subsection{Regression on the exact small configurations}

As a consistency check, the shared-kernel and hybrid constructions were
evaluated on archived exact records for the known solved square cases.  The
calculation is not used in any general proof.

\begin{table}[t]
\centering
\caption{Exact regression for the stored small-\(n\) records.  Equality in the two middle
columns refers to the stored extremal record, not to every representative or
to uniqueness of the optimum.}
\label{tab:small-regression}
\begin{tabular}{@{}c@{\qquad}c@{\qquad}c@{\qquad}c@{}}
\toprule
\(n\) & tight support saturated & silent labels & terminal residual \(W\)\\
\midrule
3 & yes & 0 & 0\\
4 & yes & 0 & 0\\
5 & yes & 0 & 0\\
6 & yes & 0 & 0\\
7 & yes & 0 & 0\\
8 & yes & 0 & 0\\
9 & yes & 0 & 0\\
\bottomrule
\end{tabular}
\end{table}

For \(n=6\), the recorded restricted determinant is \(-2/3\).  For
\(n=8\), each individual recorded stress has kernel dimension at least two,
whereas the intersection over the recorded KKT face has dimension one.  The
latter is a small but instructive instance of why the shared kernel, rather
than the kernel of a favored multiplier, is the intrinsic object.  The
machine-readable input is archived and checked by the exact verifier using
rational arithmetic.  For \(n=9\), this checks compatibility with the exact
candidate;
it does not replace the numerical \(\varepsilon\)-global certificate used in
the current literature record
\cite{MonjiModirKocuk2025,SudermannMerx2026}.  No row of
\Cref{tab:small-regression} implies a theorem for \(n\ge10\), and no
uniqueness claim is made.

\subsection{Logical consequences}

The examples enforce the following separations.
\begin{enumerate}[label=(\roman*),leftmargin=2em]
\item Tightness does not imply positive KKT support.
\item Shared-kernel cancellation does not imply cancellation for each
substress.
\item Weighted second-order neutrality does not imply rowwise flatness.
\item A smooth local determinant germ need not admit the required positive
global completion.
\item Finite small-\(n\) regression is evidence for consistency, not an
induction principle.
\end{enumerate}

These distinctions are used explicitly in the residual map below.

\section{Residual map and open problems}
\label{sec:open}

The preceding theory reduces first-order nonobservability to a concrete
finite-dimensional object
\[
 W(P^\Phi)=
 \frac{(\Ksh\oplus\Ksh)\cap\ker J_{\mathrm{tight}}^{\rm tr}}
      {\trans},
\]
chosen at a fully pinned global optimizer.  It is rank-one-free and therefore
has graph form.  This reduction is valid for every \(n\), but the existence
of a representative with \(W=0\) is not proved.

\begin{conjecture}[Existential hybrid observability]
\label{conj:hybrid-observability}
For every \(n\ge3\), the optimal level set contains a configuration \(P\)
for which
\[
 (\Ksh(P)\oplus\Ksh(P))\cap
 \ker J_{\mathrm{tight}}^{\rm tr}(P)=\trans.
\]
Equivalently, the canonical hybrid operator is positive definite on the
translation quotient at some global optimizer.
\end{conjecture}

The known reductions organize a possible counterexample by
\[
 r=\dim(\Ksh/\langle\one\rangle),
 \qquad m=\dim W.
\]
Rank-one exclusion gives
\(
 m\le2\lfloor r/2\rfloor
\).
The cases \(r\le1\) close immediately, and the literal case \(r=m=2\) is
excluded by \Cref{thm:division-residual}.  Three genuine sectors remain.

\begin{description}[leftmargin=3.5cm,style=nextline]
\item[The one-dimensional residual.]
When \(r=2,m=1\), the positive support can be flat, with
\(q_t=0\) row by row, or mixed, with nonzero coefficients of both signs.
Only the weighted equality \(\sum_t\lambda_tq_t=0\) is automatic.  In the
flat one-silent induced-five-cycle subcase, one external completion is
impossible, so any completion has depth at least two.  The open flat cases
also include covered blocks with no silent label, arbitrary noncyclic
one-silent cores and their external depth, non-induced cycles, longer
circuits, and two-silent configurations.  The mixed sector remains open as
well; this list records the principal branches rather than an exhaustive
classification of core hypergraphs.

\item[Maximal higher residuals.]
If \(0<m=r\), then \Cref{thm:maximal-silent} supplies at least \(3-c\)
stress-silent labels.  This is a lower bound, not a classification: there is
no corresponding upper bound, and the literal two-atom argument does not
extend automatically.

\item[Nonmaximal higher residuals.]
For \(1\le m<r\), the graph domain need not be invariant under its operator.
A two-dimensional slice may have useful trace and determinant identities,
but it need not contain a silent-vertex indicator or inherit the literal atom
decomposition.  This is the least rigid present branch.
\end{description}

The separator formalism suggests three focused questions.

\begin{question}
Can square-boundary moments force a definite determinant sign on every
one-dimensional mixed residual, or otherwise rule out weighted cancellation
between positive rows?
\end{question}

\begin{question}
Can one select an invariant division plane inside a higher-dimensional
rank-one-free residual while retaining enough boundary information to bound
the number of silent labels from above?
\end{question}

\begin{question}
Can the two-terminal flux law be iterated without assuming internally
vertex-disjoint isolated blocks, perhaps through an additive or submodular
boundary budget?
\end{question}

These questions are deliberately narrower than the original extremal
problem.  Their common theme is to convert a weighted global cancellation
into local geometric information without assuming strict complementarity.
The same stress and affine-moment mechanism is available in other max--min
determinant problems; higher-dimensional simplex volumes should have an
analogous exterior-algebra formulation, although no such theory is developed
here.

\subsection*{Conclusion}

Optimal Heilbronn configurations carry more structure than generic KKT
points.  Their determinant multipliers form skew stresses with isotropic
affine moments; positive components are constrained by all four sides of the
square; nonunique multipliers lead naturally to a shared kernel and a
canonical hybrid observability operator; and two-terminal pieces transmit an
explicit flux, represented by Pfaffian cofactors in the five-internal-vertex
rank-four case.  A terminal selection turns the remaining
degeneracy into a rank-one-free matrix space, and exact computation excludes
one minimal global completion mechanism.  What remains is sharply stated by
\Cref{conj:hybrid-observability}.  The full Heilbronn triangle problem for
arbitrary \(n\) remains open.

\section*{Data and code availability}

The exact symbolic inputs, verification programs, expected outputs, declared
scope conditions, and checksum manifest used for the finite claims are
included with the supplementary material.  Verification uses Python 3.10 or
later and SymPy 1.14.0, runs without network access after installation of the
declared dependency, and stops at the first failed or malformed certificate.
The supplementary archive supplied with this manuscript is the presently
available release; a permanent repository identifier may be added to the
version of record.

\appendix
\section{Normalization, factors, and signs}
\label{app:factors}

This appendix fixes the conventions used throughout and gives short
coordinate checks for the identities most sensitive to a factor of two or a
sign.

Let \(t=(i,j,k)\), with \(i<j<k\), and put
\[
 a_t=e_j-e_i,\qquad c_t=e_k-e_i,\qquad
 B_t=s_t(a_tc_t^T-c_ta_t^T),
\]
where
\(
 s_t=\operatorname{sgn}\det(p_j-p_i,p_k-p_i)
\).
Then \(B_t^T=-B_t\), \(B_t\one=0\), and
\begin{equation}
 X^TB_tY
 =s_t\det(p_j-p_i,p_k-p_i)=2A_t.     \label{eq:app-area}
\end{equation}
Ordinary area, not doubled area, is used everywhere.

For \(B=\sum_t\lambda_tB_t\), direct differentiation gives
\[
 \frac{\partial}{\partial X}\sum_t\lambda_tA_t
 =\frac12BY,
 \qquad
 \frac{\partial}{\partial Y}\sum_t\lambda_tA_t
 =-\frac12BX.
\]
With outward square load
\[
 b_i=(r_i-\ell_i)+\mathrm i(u_i-d_i),
\]
position stationarity therefore reads
\[
 BY=2\operatorname{Re}b,
 \qquad -BX=2\operatorname{Im}b,
 \qquad Bz=2\mathrm i b.             \label{eq:app-Bz}
\]

For a tight subfamily \(\mathcal H\) carrying the weights inherited from
one actual multiplier, define
\(
 B_{\mathcal H}z=2\mathrm i b^{\mathcal H}
\)
and
\(
 m_{\mathcal H}=\sum_{t\in\mathcal H}\lambda_t
\).
Skew symmetry and \eqref{eq:app-area} give
\begin{align}
 \sum_i p_i(b_i^{\mathcal H})^T
 &=
 \begin{pmatrix}
 \tfrac12X^TB_{\mathcal H}Y&-\tfrac12X^TB_{\mathcal H}X\\
 \tfrac12Y^TB_{\mathcal H}Y&-\tfrac12Y^TB_{\mathcal H}X
 \end{pmatrix} \notag\\
 &=m_{\mathcal H}\Delta I_2.        \label{eq:app-partial-moment}
\end{align}
The vector \(b^{\mathcal H}\) is an induced algebraic load.  It agrees with
a portion of the physical square reaction only under the isolation
hypotheses stated in \Cref{sec:separator}.

For a shadow pair \(Q_i=(\xi_i,\eta_i)\),
\[
 \xi^TB_t\eta
 =s_t\det(Q_j-Q_i,Q_k-Q_i).
\]
If \(t\) is tight and
\(
 C_t=E_t^Q(E_t^P)^{-1}
\), then
\begin{equation}
 \xi^TB_t\eta=2\Delta\det C_t.       \label{eq:app-shadow-factor}
\end{equation}
The sign in the two-terminal law follows from
\[
 \sum_i\xi_i b_{i,x}^{\mathcal H}
 =\tfrac12\xi^TB_{\mathcal H}Y
 =-\tfrac12(B_{\mathcal H}\xi)^TY,
\]
and
\[
 \sum_i\xi_i b_{i,y}^{\mathcal H}
 =-\tfrac12\xi^TB_{\mathcal H}X
 =\tfrac12(B_{\mathcal H}\xi)^TX.
\]
Thus, if
\(
 B_{\mathcal H}\xi=\rho(e_a-e_b)
\),
\begin{equation}
 \sum_i\xi_i(b_i^{\mathcal H})^T
 =\frac{\rho}{2}(Y_b-Y_a,\,X_a-X_b). \label{eq:app-terminal-sign}
\end{equation}

\begin{table}[t]
\centering
\caption{Convention audit.}
\label{tab:factor-audit}
\begin{tabularx}{\textwidth}{@{}lXX@{}}
\toprule
Object & Convention & Consequence\\
\midrule
Triangle area & \(A_t=|D_t|/2\) & \(X^TB_tY=2A_t\)\\
Square reaction & outward normal & \(Bz=2\mathrm i b\)\\
Partial load & \(B_{\mathcal H}z=2\mathrm i b^{\mathcal H}\) &
moment \(m_{\mathcal H}\Delta I_2\)\\
Shadow coefficient & Hessian quadratic value &
\(q_t=\xi^TB_t\eta=2\Delta\det C_t\)\\
Terminal residual & \(\rho(e_a-e_b)\) &
right side of \eqref{eq:app-terminal-sign}\\
\bottomrule
\end{tabularx}
\end{table}

\section{Rank-one-free planes and the unique-apex step}
\label{app:rank-one-free}

We expand the only calculation suppressed in the proof of
\Cref{thm:division-residual}.  In the literal case
\(
 \dim H_0=\dim W=2
\), the support-collapse argument gives one positive component \(C\) and
two interior silent labels \(a,b\).  After subtracting translations, every
residual velocity vanishes on \(C\), and evaluation at either silent label is
an isomorphism.  Hence there is a matrix \(T\in GL_2(\R)\) such that
\begin{equation}
 v_a=u,\qquad v_b=Tu,\qquad v_j=0\quad(j\in C),
 \qquad u\in\R^2.                    \label{eq:app-two-silent}
\end{equation}
The rank-one-free property says that \(T\) has no real eigenvector.  In
particular, \(T-I\) is invertible.

Let
\(
 \Omega=\left(\begin{smallmatrix}0&-1\\1&0\end{smallmatrix}\right)
\), so that \(\det(x,y)=-x^T\Omega y\).  For \(j\in C\), differentiate the
oriented double area
\(
 D(a,b,j)=\det(p_b-p_a,p_j-p_a)
\)
along \eqref{eq:app-two-silent}.  Since \(v_j=0\), the derivative is
\begin{align*}
 0
 &=\det((T-I)u,p_j-p_a)
   +\det(p_b-p_a,-u)\\
 &=-u^T\bigl((T-I)^T\Omega(p_j-p_a)
              +\Omega(p_b-p_a)\bigr).
\end{align*}
First-order invisibility holds for every \(u\), hence
\begin{equation}
 (T-I)^T\Omega(p_j-p_a)+\Omega(p_b-p_a)=0.
                                                        \label{eq:app-apex}
\end{equation}
If both \((a,b,j)\) and \((a,b,j')\) are tight, subtracting their instances
of \eqref{eq:app-apex} gives
\[
 (T-I)^T\Omega(p_j-p_{j'})=0.
\]
Both factors on the left are invertible, so \(p_j=p_{j'}\), which is
impossible when the minimum triangle area is positive.  The atom-type
elimination in the proof of \Cref{thm:division-residual} has already shown
that every tight triangle incident with \(a\) has type \(abC\).  Therefore at
most one tight triangle is incident with \(a\).  With no boundary normal at
this interior point, the active normals of its three-variable one-point
linear program have rank at most two, contradicting full pinning.

The calculation is literal.  If \(W\) is only a two-dimensional slice of a
larger residual, evaluation at the two silent labels need not be an
isomorphism and the atom equality used before \eqref{eq:app-two-silent} need
not hold.

For completeness, the determinant-definiteness lemma can be expressed
without a normal-form choice.  If \(\mathcal D\subset M_2(\R)\) is a
two-dimensional division plane and \(M_0\in\mathcal D\setminus\{0\}\),
right multiplication by \(M_0^{-1}\) takes it to a plane containing \(I\).
After removing the scalar part of a second generator, its characteristic
polynomial has negative discriminant.  Conjugation and scaling reduce it to
\(\mathcal J^2=-I\).  The trace annihilator then consists of
\(
 \left(\begin{smallmatrix}x&y\\y&-x\end{smallmatrix}\right)
\), on which determinant equals \(-x^2-y^2\).  Undoing the transformations
multiplies determinant by one fixed nonzero scalar, establishing the fixed
sign used in \Cref{thm:division-residual,thm:maximal-silent}.

\section{Exact computation and reproducibility}
\label{app:computation}

The finite results in \Cref{sec:c5,sec:limits} are supplied as a separate
replay package.  Its role is limited: the conceptual reduction to finitely
many families is proved in the text, while exact algebra decides the residual
families.  No random sampling or floating-point tolerance determines a
mathematical verdict.

The package contains seven certificate families:
\begin{enumerate}[leftmargin=2em]
\item the two-terminal load identity and the five-vertex Pfaffian carrier;
\item the isolation-qualified separator reaction and block counts;
\item the strict local five-cycle survivor and the same-label endpoint scan;
\item the all-orientation one-external exclusion;
\item the generic selector counterexample;
\item the midpoint partial-flux counterexample;
\item the finite solved-case regression for \(3\le n\le9\).
\end{enumerate}

Polynomial identities are checked by coefficient equality.  Rational
inequalities use exact fractions; algebraic signs in
\(\mathbb Q(\sqrt{1749})\) are certified by rational isolating bounds.
Jacobian ranks are decided by exact minors.  Generic five-cycle charts are
reduced by rational-function elimination, while denominator-zero cases are
reintroduced as separate polynomial systems.  Gr\"obner bases and exact real
root tests decide the remaining endpoints.  Inactive triangle inequalities
are tested after every algebraic candidate is recovered.

The all-orientation calculation first quotients by the dihedral action and
global sign reversal.  Four representatives, with orbit sizes
\(2,10,10,10\), cover all \(32\) words.  The primary enumeration checks
\(480\) generic patterns and \(1360\) singular charts.  A separate ungauged
exceptional-rank implementation checks \(1440\) cases.  Agreement of the two
routes is part of the certificate rather than an informal spot check.

The verification program stops on a missing dependency, malformed
certificate, failed assertion, or violated declared scope condition.  It
also checks a manifest of cryptographic hashes.  Python 3.10 or later is
required; only the local five-cycle calculation additionally requires SymPy.
After dependencies are installed, replay is self-contained and requires no
network access.  The source archive includes a claim-to-certificate map,
machine-readable receipts, and exact input data.

Computational conclusions are stated at their verified realizability level.
In particular, the strict five-cycle survivor is a local determinant-and-trace
configuration rather than a global KKT optimizer; the midpoint example is a
four-label partial-flux counterexample; and the solved-case table is a finite
regression test rather than evidence for an induction.

\section*{Acknowledgements}
OpenAI ChatGPT was used as an assistive tool during the development and
preparation of this manuscript, including literature-search assistance,
organization of research material, drafting and language editing, and
assistance with symbolic-computation code and proof checking. The author
reviewed and verified all mathematical statements, proofs, citations,
computations, code outputs, and final wording, and takes full responsibility
for the content of the work.

\section*{Statements and Declarations}
\textbf{Funding.} No funding was received for conducting this study or
preparing this manuscript.

\textbf{Competing interests.} The author declares no competing financial or
non-financial interests.

\textbf{Author contributions.} Dawid Trela was responsible for
conceptualization, methodology, formal analysis, investigation, software,
validation, visualization, and writing (original draft, review, and editing).

\textbf{Data and code availability.} The exact symbolic inputs, verification
programs, expected outputs, declared scope conditions, and SHA-256 manifest
supporting the finite computational claims are included in Online Resource~1
supplied with the submission. Verification uses exact arithmetic, requires
Python~3.10 or later and SymPy~1.14.0, and runs without network access after
dependency installation.

\textbf{Ethics approval and consent to participate.} Not applicable.

\textbf{Consent for publication.} Not applicable.

\bibliographystyle{plain}
\bibliography{references}

\end{document}